\documentclass[12pt]{amsart}
\usepackage{amsmath,amsfonts,amssymb,latexsym}
\usepackage{hyperref}
\usepackage{graphicx}
\usepackage[a4paper, left=2cm, right=2cm, top=2cm, bottom=2cm]{geometry}

\newcommand{\NN}{\mathbb{N}}
\newcommand{\CC}{\mathbb{C}}
\newcommand{\RR}{\mathbb{R}}
\newcommand{\QQ}{\mathbb{Q}}
\newcommand{\ZZ}{\mathbb{Z}}
\newcommand{\EE}{\mathbb{E}}
\newcommand{\Oo}{\mathcal{O}}
\newcommand{\Bo}{\mathcal{B}}
\newcommand{\RE}{ {\rm Re \,} }

\newtheorem{Thm}{Theorem}
\newtheorem{Lem}{Lemma}
\newtheorem{Prop}{Proposition}
\newtheorem{Cor}{Corollary}
\theoremstyle{definition}
\newtheorem{Rem}{Remark}
\newtheorem{Def}{Definition}
\newtheorem{Ex}{Example}

\begin{document}
\title[Linear $q$-difference-differential equations with constant coefficients]{On the summability and convergence of formal solutions of linear $q$-difference-differential equations with constant coefficients}
\author{Kunio Ichinobe}
\address{Department of Mathematics Education,
Aichi University of Education,
1 Hirosawa, Igaya, Kariya City, Aichi Prefecture 448-8542, Japan}
\email{ichinobe@auecc.aichi-edu.ac.jp}
  \author{S{\l}awomir Michalik}
  \address{Faculty of Mathematics and Natural Sciences,
College of Science,
Cardinal Stefan Wyszy\'nski University,
W\'oycickiego 1/3,
01-938 Warszawa, Poland,
ORCiD: 0000-0003-4045-9548}
\email{s.michalik@uksw.edu.pl}
\urladdr{\url{http://www.impan.pl/~slawek}}
\begin{abstract}
 We consider the Cauchy problem for homogeneous linear $q$-difference-differential equations with constant coefficients. We characterise convergent, $k$-summable and multisummable formal power series solutions in terms of analytic continuation properties and growth estimates of the Cauchy data. We also introduce and characterise sequences preserving summability, which make a very useful tool, especially in the context of moment differential equations.  
\end{abstract}

\maketitle
\section{Introduction}
The main purpose of the paper is the characterisation of summable formal power series solutions of general homogeneous linear $q$-difference-differential equations with constant coefficients in the framework of the theory of Gevrey asymptotics and $k$-summability. More precisely, we consider the Cauchy problem
\begin{equation}
\label{eq:CP}
\left\{
    \begin{array}{l}
     P(D_{q,t},\partial_z)u=0,\\  
     D^j_{q,t}u(0,z)=\varphi_j(z)\ \text{for}\ j=0,\dots,p-1,
     \end{array}
    \right.
\end{equation}
where $(t,z)\in\CC^2$, $\varphi_j(z)$, $j=0,\dots,p-1$, are holomorphic functions in the complex neighbourhood of the origin and $P(D_{q,t},\partial_z)$ is a general linear $q$-difference-differential operator with constant coefficients of order $p$ with respect to $D_{q,t}$. For given $q\neq 1$, the $q$-difference operator $D_{q,t}$ is defined by 
\begin{equation*}
 D_{q,t}u(t,z):=\frac{u(qt,z)-u(t,z)}{qt-t}.
\end{equation*}

The formal solutions of $q$-difference-differential equations were studied in such papers as \cite{L-M-S,Mal,Tah,T-Y}, but only from the point of view $q$-asymptotics and $q$-Borel summability, and under the assumption $|q|>1$. In this situation the coefficients of formal solutions are estimated by powers of $q$.

In our different approach we assume that $q\in[0,1)$ and we study the Gevrey asymptotics and $k$-summability, where the coefficients of formal solutions are estimated by the gamma function.
This approach was previously studied only by Ichinobe and Adachi \cite{I-A} in a very special case of the formal solutions of the Cauchy problem 
\begin{equation}
\label{eq:CP_special}
(D^{\kappa}_{q,t}-\partial_z^{\nu})u=0,\quad u(0,z)=\varphi(z),\quad D^j_{q,t}u(0,z)=0\quad (j=1,\dots,\kappa-1),
\end{equation}
where $(t,z)\in\CC^2$, $\varphi(z)$ is holomorphic in a complex neighbourhood of the origin and $\kappa,\nu\in\NN$.
They have characterised $k$-summability of the formal solution $\hat{u}$ of (\ref{eq:CP_special})
in terms of analytic continuation property and growth estimate of the Cauchy datum $\varphi(z)$.

We generalise the results of \cite{I-A} to the formal solutions of (\ref{eq:CP}). The main idea of the paper is to prove that the sequence $m=(m(n))_{n\geq 0}=([n]_q!)_{n\geq 0}$ preserves summability in the sense that $\hat{u}(t,z)=\sum_{n=0}^{\infty}\frac{u_n(z)}{m(n)}t^n$ is $k$-summable in a direction $d$ if and only if $\hat{v}(t,z)=\sum_{n=0}^{\infty}u_n(z)t^n$ is also $k$-summable in the same direction.
The proof of the equivalence  is based on the moment version of the result of Ichinobe and Adachi \cite{I-A} for the equation (\ref{eq:CP_special}) with $\kappa=\nu=1$.
Using this equivalence and observing that for the sequence $m=([n]_q!)_{n\geq 0}$ the $q$-difference operator $D_{q,t}$ coincides with the $m$-moment differential operator $\partial_{m,t}$ defined by
\begin{equation*}
 \partial_{m,t}\Big(\sum_{n=0}^{\infty}\frac{u_n(z)}{m(n)}t^n\Big):=\sum_{n=0}^{\infty}\frac{u_{n+1}(z)}{m(n)}t^n,
\end{equation*}
we conclude that the solution $\hat{u}(t,z)=\sum_{n=0}^{\infty}\frac{u_n(z)}{m(n)}t^n$ of (\ref{eq:CP}) is $k$-summable in a given direction $d$ if and only if the solution $\hat{v}(t,z)=\sum_{n=0}^{\infty}u_n(z)t^n$ of the initial value problem
\begin{equation}
\label{eq:CPv}
\left\{
    \begin{array}{l}
     P(\partial_{\mathbf{1},t},\partial_z)v=0,\\  
     \partial^j_{\mathbf{1},t}v(0,z)=\varphi_j(z)\ \text{for}\ j=0,\dots,p-1,
     \end{array}
    \right.
\end{equation}
is $k$-summable in the same direction. Here $\partial_{\mathbf{1},t}$ denotes an $\tilde{m}$-moment differential operator for the sequence $\tilde{m}=\mathbf{1}:=(1)_{n\geq 0}$, so the operator $\partial_{\mathbf{1},t}$ satisfies
\begin{equation*}
 \partial_{\mathbf{1},t}\hat{v}(t,z)=D_{0,t}\hat{v}(t,z)=\frac{\hat{v}(t,z)-\hat{v}(0,z)}{t}.
\end{equation*}

Moreover, since $\mathbf{1}=(1)_{n\geq 0}$ is a sequence of moments for the moment function $\mathbf{1}(u)\equiv 1$ for $u\geq 0$,
we may characterise summable and multisummable solutions $\hat{v}(t,z)$ of (\ref{eq:CPv}) using the whole theory of formal solutions of moment differential equations developed in \cite{Mic7} and \cite{Mic8}, and in consequence also the summable and multisummable solutions $\hat{u}(t,z)$ of (\ref{eq:CP}) in terms of analytic continuation properties and growth estimates of the Cauchy data $\varphi_j(z)$, $j=0,\dots,p-1$.

We would like to emphasize that, as introduced in the paper, the idea of sequences preserving summability seems interesting in and of itself. In the paper we start to develop this theory. In particular we show that the family of sequences preserving summability contains the sequences of moments of
order $0$ in the sense of Balser's theory of moment summability \cite{B2}. On the other hand this family is contained sharply
in the family of sequences of positive numbers of
order $0$. Moreover, the sequences preserving summability form a group with respect to the multiplication.
We also find the characterisation of this kind of sequences.

In the paper we apply a new idea of use of auxiliary Cauchy problems for the simple moment equations $(\partial_{m_1,t}-\partial_{m_2,z})u(t,z)=0$, where we study the relation between the initial data $u(0,z)$ and the boundary data $u(t,0)$.
 The auxiliary Cauchy problems and their solutions give us strong tools, which transform the power series $z\mapsto u(0,z)$ onto the power series $t\mapsto u(t,0)$. These tools are intensively used in the proofs of Theorems \ref{th:sequence_summability} and \ref{th:main}.

The paper is organised as follows. In the preliminary sections \ref{sec:2}--\ref{sec:5}, we collect notation and recall the notion of $q$-calculus, Gevrey order, summability, multisummability and moment functions. 
In sections \ref{sec:6} and \ref{sec:7}, we introduce new concepts of sequences, which preserve Gevrey order or summability. Moreover, we find the characterisation of such types of sequences.
We also define $m$-moment differentiation and operators of order zero. 
In section \ref{sec:8}, we prove the main result of the paper, which says that the sequence $([n]_q!)_{n\geq 0}$ preserves summability. The proof is based on a few lemmas about analytic continuation properties of the initial and boundary data of the solutions to some auxiliary moment differential equations. 
In the final section \ref{sec:9}, we characterise summable and multisummable solutions to the equation $P(\partial_{m,t},\partial_z)u=0$, where $m$ is a sequence preserving summability. In particular, it gives such characterisation for the formal power series solutions of the $q$-difference-differential equation $P(D_{q,t},\partial_z)u=0$.

\section{Notation} \label{sec:2} 
An~\emph{unbounded sector $S$ in a direction $d\in\RR$ with an opening $\alpha>0$} in the universal covering space $\widetilde{\CC}$ of
$\CC\setminus\{0\}$ is defined by
\begin{displaymath}
S=S_d(\alpha):=\{z\in\widetilde{\CC}\colon\ z=r\*e^{i\phi},\ r>0,\ \phi\in(d-\alpha/2,d+\alpha/2)\}.
\end{displaymath}
If the opening $\alpha$ is not essential,
 the sector $S_d(\alpha)$ is denoted briefly by $S_d$.
 
 A complex disc $D_r$ in $\CC$ with a radius $r>0$ is a set of the form 
 \begin{displaymath}
 D_r=\{z\in\CC:|z|<r\}.
 \end{displaymath}
 In case that the radius $r$ is not essential, the set $D_r$ will be designated  briefly by $D$. We also denote briefly a \emph{disc-sector}
 $S_d(\alpha)\cup D_r$ (resp. $S_d(\alpha)\cup D$, $S_d\cup D$) by $\hat{S}_d(\alpha,r)$ (resp. $\hat{S}_d(\alpha)$, $\hat{S}_d$).
 
If a function $f$ is holomorphic on a domain $U\subseteq\CC^n$, then it will be denoted by $f\in\Oo(U)$.

Analogously,
the space of holomorphic functions of the variables $z_1^{1/\kappa_{1}},\dots,z_n^{1/\kappa_n}$
($(\kappa_1,\dots,\kappa_n)\in\NN^n$) on $G$ is denoted by $\mathcal{O}_{1/\kappa_{1},\dots,1/\kappa_n}(G)$.

More generally, if $\EE$ denotes a Banach space with a norm $\|\cdot\|_{\EE}$, then
by $\Oo(G,\EE)$ (resp. $\Oo_{1/\kappa_{1},\dots,1/\kappa_n}(G,\EE)$) we shall denote the set of all $\EE$-valued  
holomorphic functions 
(resp. holomorphic functions of the variables $z_1^{1/\kappa_{1}},\dots,z_n^{1/\kappa_n}$) 
on a domain $G\subseteq\CC^n$.
For more information about functions with values in Banach spaces we refer the reader to \cite[Appendix B]{B2}. 
In the paper, as a Banach space $\EE$ we will
take the space of complex numbers $\CC$ (we abbreviate $\Oo(G,\CC)$ to $\Oo(G)$ and
$\Oo_{1/\kappa_{1},\dots,1/\kappa_n}(G,\CC)$ to $\Oo_{1/\kappa_{1},\dots,1/\kappa_n}(G)$)
or the space $\Oo_{1/\kappa}(D_r)\cap C(\overline{D}_r)$ of $1/\kappa$-holomorphic functions on the disc $D_r$ and continuous on its closure $\overline{D}_r$, equipped with the norm $\|\varphi\|_r:=\sup_{z\in D_r}|\varphi(z)|$.

The space of formal power series
$ \hat{u}(t)=\sum_{j=0}^{\infty}u_j t^{j}$ with $u_j\in\mathbb{E}$ is denoted by $\mathbb{E}[[t]]$.

We use the ``hat'' notation ($\hat{u}$, $\hat{v}$, $\hat{f}$, $\hat{\varphi}$) to denote the formal power series.
If the formal power series $\hat{u}$ (resp. $\hat{v}$, $\hat{f}$, $\hat{\varphi}$) is convergent,
we denote its sum by $u$ (resp. $v$, $f$, $\varphi$).

Throughout the paper we will assume that $m=(m(n))_{n\geq 0}$ (resp. $m_1=(m_1(n))_{n\geq 0}$, $m_2=(m_2(n))_{n\geq 0}$) is a sequence of real positive numbers with $m(0)=1$ (resp. $m_1(0)=1$, $m_2(0)=1$). Moreover, for fixed $k>0$ we denote by $\Gamma_{1/k}$ the sequence of positive numbers $\Gamma_{1/k}=(\Gamma(1+n/k))_{n\geq0}$, where $\Gamma(\cdot)$ is the gamma function. 

\section{$q$-calculus} \label{sec:3} 
In this section we introduce the basic notion of $q$-calculus following \cite{G-R}. Throughout the whole paper, we assume that  $q\in[0,1)$. 
If $u\in\EE[[t]]$ or $u\in\Oo(D,\EE)$ then we define the \emph{$q$-difference operator} $D_{q,t}$ as
\begin{equation*}
D_{q,t}u(t):=\frac{u(qt)-u(t)}{qt-t}.
\end{equation*}
For every $n\in\NN_{0}$ we define a $q$-analog of $n$ by 
\begin{equation*}
 [n]_q:=1+q+\dots+q^{n-1}=\frac{1-q^n}{1-q}.
\end{equation*}
We also introduce a $q$-analog of the factorial $n!$
\begin{equation*}
 [n]_q!:=
    \left\{
    \begin{array}{ll}
     1&\textrm{for}\ n=0\\  
     {[1]_q\cdots [n]_q}&\text{for}\ n\geq 1
     \end{array}
    \right..
\end{equation*}
 For every $n\in\NN_0\cup\{\infty\}$ and $a\in\CC$ we define $q$-shift factorial by
\begin{equation*}
 (a;q)_n:=
    \left\{
    \begin{array}{ll}
     1&\text{for}\ n=0\\
     \prod_{j=0}^{n-1}(1-aq^j)&\text{for}\ n\in\NN\\
     \prod_{j=0}^{\infty}(1-aq^j)&\text{for}\ n=\infty
     \end{array}
    \right..
\end{equation*}
Observe that the infinity product $(a;q)_{\infty}$ is convergent for $|q|<1$ and for any $a\in\CC$. We also set 
\begin{equation*}
 (a_1,\dots,a_r;q)_n:=\prod_{j=1}^r(a_j;q)_n\quad\text{for}\quad a_1,\dots,a_r\in\CC\quad\textrm{and}\quad n\in\NN_0\cup\{\infty\}.
\end{equation*}
The basic hypergeometric series is defined as
\begin{equation}
\label{eq:q_binomial}
 {}_{k+1}\phi_k\left(\begin{array}{c}
                       a_1,\dots,a_{k+1}\\
                       b_1,\dots,b_k
                     \end{array}
;q,x\right):=\sum_{n=0}^{\infty}\frac{(a_1,\dots,a_{k+1};q)_n}{(b_1,\dots,b_k;q)_n(q;q)_n}x^n.
\end{equation}

In the paper we will use the following two fundamental formulas for the basic hypergeometric series
\begin{Prop}[$q$-binomial theorem, {\cite[Section 1.3]{G-R}}]
\label{prop:1}
\begin{equation*} 
 {}_{1}\phi_0\left(\begin{array}{c}
                       a\\
                       -
                     \end{array}
;q,z\right)=\sum_{n=0}^{\infty}\frac{(a;q)_n}{(q;q)_n}z^n=\frac{(az;q)_{\infty}}{(z;q)_{\infty}}\quad\textrm{for}\quad |z|<1\quad\textrm{and}\quad |q|<1.
\end{equation*}
\end{Prop}

\begin{Prop}[Heine's transformation formula, {\cite[Section 1.4]{G-R}}]
\label{prop:2}
\begin{equation*} 
 {}_{2}\phi_1\left(\begin{array}{c}
                       a,b\\
                       c
                     \end{array}
;q,z\right)=\sum_{n=0}^{\infty}\frac{(a,b;q)_n}{(c;q)_n(q;q)_n}z^n=\frac{(b,az;q)_{\infty}}{(c,z;q)_{\infty}}{}_{2}\phi_1\left(\begin{array}{c}
                       c/b,z\\
                       az
                     \end{array}
;q,b\right),\ |z|<1,\, |q|<1,\, |b|<1.
\end{equation*}
\end{Prop}

\section{Gevrey order, summability and multisummability} \label{sec:4} 
In this section we introduce some definitions and fundamental facts connected with a Gevrey order, $k$-summability and multisummability. For more details we refer the reader to \cite{B2} and \cite{LR}.
\begin{Def}
\label{df:growth}
A function $u\in\Oo_{1/\kappa}(\hat{S}_d(\varepsilon,r),\EE)$
is of \emph{exponential growth of order at most $k\in\RR$ as $x\to\infty$
in $\hat{S}_d(\varepsilon,r)$} if for any
$\tilde{\varepsilon}\in(0,\varepsilon)$ and $\tilde{r}\in(0,r)$ there exist 
$A,B>0$ such that
\begin{gather*}
\|u(x)\|_{\EE}<Ae^{B|x|^k} \quad \textrm{for every} \quad x\in \hat{S}_d(\tilde{\varepsilon},\tilde{r}).
\end{gather*}
The space of such functions is denoted by $\Oo_{1/\kappa}^k(\hat{S}_d(\varepsilon,r),\EE)$, or $\Oo^k(\hat{S}_d(\varepsilon,r),\EE)$ if $\kappa=1$.
\end{Def}

\begin{Def}
Let $s\in\RR$. If the sequence $m=(m(n))_{n\geq 0}$ satisfies the condition:
\begin{equation}
\label{eq:sequence_order}
 \text{there exist}\ a,A>0\ \text{such that}\quad a^n(n!)^s\leq m(n) \leq A^n (n!)^s\quad \text{for every}\ n\in\NN_0,
\end{equation}
then $m$ is called a \emph{sequence of order $s$}.
\end{Def}

\begin{Ex}
Let $k>0$. 
The sequence $\Gamma_{1/k}=(\Gamma(1+n/k))_{n\ge0}$ is a sequence of order $1/k$.
\end{Ex}

\begin{Def}[see {\cite[Section 5.2]{B2}}]
 For a fixed sequence $m=(m(n))_{n\geq 0}$ of positive numbers with $m(0)=1$, a linear operator $\Bo_{m,t}\colon \EE[[t]]\to\EE[[t]]$ defined by
 \begin{equation*}
  (\Bo_{m,t}\hat{u})(t):=
  \sum_{n=0}^{\infty}\frac{a_n}{m(n)}t^n\quad\text{for}\quad\hat{u}(t)=\sum_{n=0}^{\infty}a_nt^n\in\EE[[t]]
 \end{equation*}
 is called an \emph{$m$-Borel operator with respect to $t$}.
\end{Def}
\begin{Rem}
 Observe that for a given sequence $m=(m(n))_{n\geq 0}$ of positive numbers with $m(0)=1$, an inverse $m$-Borel operator $\Bo_{m,t}^{-1}\colon\EE[[t]]\to\EE[[t]]$, called sometimes an \emph{$m$-Laplace operator} is given by
 $\Bo_{m,t}^{-1}= \Bo_{m^{-1},t}$ on $\EE[[t]]$, where $m^{-1}=(m(n)^{-1})_{n\geq 0}$. 
 Hence an $m$-Borel operator $\Bo_{m,t}$ is a linear automorphism on the space of formal power series $\EE[[t]]$.
\end{Rem}

\begin{Def}
    \label{df:gevrey}
    Let $s\in\RR$. A series
    $\hat{u}(t)=\sum_{n=0}^{\infty}a_nt^n\in\EE[[t]]$ is called a \emph{formal power series of Gevrey order $s$} if
    there exist constants $B,C<\infty$ such that
    \begin{equation}
    \label{eq:est_un}
     \|a_n\|_{\EE}\leq BC^n(n!)^s\quad\textrm{for every}\quad n\in\NN_0.
    \end{equation}
    The space of formal power series of Gevrey 
    order $s$ is denoted by $\EE[[t]]_s$.
\end{Def}

\begin{Rem}
    \label{re:gevrey}
    Let $s\in\RR$ and $m$ be a sequence of order $s$. By the above definitions $\hat{u}\in\EE[[t]]_s$ if and only if
    there exists a disc $D\subseteq\CC$ with centre at the origin such that
    ${\Bo}_{m,t}\hat{u}\in\Oo(D,\EE)$.
\end{Rem}

\begin{Def}
\label{df:summable}
Let $k>0$ and $d\in\RR$. Then $\hat{u}\in\EE[[t]]$ is called
\emph{$k$-summable in a direction $d$} if there exists a disc-sector $\hat{S}_d$ in a direction $d$ such that
$\Bo_{\Gamma_{1/k},t}\hat{u}\in\Oo^k(\hat{S}_d,\EE)$.

The space of $k$-summable formal power series in a direction $d$ is denoted by $\EE\{t\}_{k,d}$.
\end{Def}

We will consider also the larger class of summable series than $k$-summable ones, called multisummable ones, for which we may also get a unique sum on appropriate sector as in the $k$-summable case (see \cite[Section 10]{B2} and \cite[Section 7]{LR}).     
\begin{Def}
\label{df:multisummable}
 Let $k_1>\cdots>k_N>0$. We say that a real vector $(d_1,\dots,d_N)\in\RR^N$ is an
 \emph{admissible multidirection} if
 \begin{gather*}
  |d_j-d_{j-1}| \leq \pi(1/k_j - 1/k_{j-1})/2 \quad \textrm{for} \quad j=2,\dots,N.
 \end{gather*}
 \par
 Let $\mathbf{k}=(k_1,\dots,k_N)\in\RR^N_+$ and  let $\mathbf{d}=(d_1,\dots,d_N)\in\RR^N$ be an
 admissible multidirection.
 We say that a formal power series
 $\hat{u}\in\EE[[t]]$ is {\em $\mathbf{k}$-multisummable in the
 multidirection $\mathbf{d}$}
 if $\hat{u}=\hat{u}_1+\cdots+\hat{u}_N$, where $\hat{u}_j\in\EE[[t]]$ is
 $k_j$-summable
 in the direction $d_j$ for $j=1,\dots,N$.
 
 The space of $\mathbf{k}$-multisummable formal power series in a multidirection $\mathbf{d}$ is denoted by $\EE\{t\}_{\mathbf{k},\mathbf{d}}$.
\end{Def}

\section{Moment functions} \label{sec:5} 
In this section we introduce the set of sequences of moments
\begin{equation}
\label{eq:M_got}
\mathfrak{M}:=\{(\mathfrak{m}(n))_{n\geq 0}\colon \mathfrak{m}(u)\ \textrm{is a moment function}\}
\end{equation}
and we show that it is a subset
of the set
\begin{equation}
 \label{eq:M_cal}
 \mathcal{M}:=\{(m(n))_{n\geq 0}\colon m(0)=1,\ m(n)> 0\ \textrm{for}\ n\in\NN\}
\end{equation}
of all positive sequences $m=(m(n))_{n\geq 0}$ with $m(0)=1$. 

To introduce moment functions we recall the notion of moment methods found by Balser~\cite{B2} (see also \cite{B0}).
   
   \begin{Def}[see {\cite[Section 5.5]{B2}}]
    \label{df:moment}
    A pair of functions $e_\mathfrak{m}$ and $E_\mathfrak{m}$ is said to be \emph{kernel functions of order $k$} ($k>1/2$) if
    they have the following properties:
   \begin{enumerate}
    \item[1.] $e_\mathfrak{m}\in\Oo(S_0(\pi/k))$, $e_\mathfrak{m}(z)/z$ is integrable at the origin, $e_\mathfrak{m}(x)\in\RR_+$ for $x\in\RR_+$ and
     $e_\mathfrak{m}$ is exponentially flat of order $k$ in $S_0(\pi/k)$ (i.e. for every $\varepsilon > 0$ there exists $A,B > 0$
     such that $|e_\mathfrak{m}(z)|\leq A e^{-(|z|/B)^k}$ for $z\in S_0(\pi/k-\varepsilon)$).
    \item[2.] $E_\mathfrak{m}\in\Oo^{k}(\CC)$ and $E_\mathfrak{m}(1/z)/z$ is integrable at the origin in $S_{\pi}(2\pi-\pi/k)$.
    \item[3.] The connection between $e_\mathfrak{m}$ and $E_\mathfrak{m}$ is given by the corresponding \emph{moment function
    $\mathfrak{m}$ of order $1/k$} as follows.
     The function $\mathfrak{m}$ is defined in terms of $e_\mathfrak{m}$ by
     \begin{gather}
      \label{eq:e_m}
      \mathfrak{m}(u):=\int_0^{\infty}x^{u-1} e_\mathfrak{m}(x)dx \quad \textrm{for} \quad \RE u \geq 0
     \end{gather}
     and the kernel function $E_\mathfrak{m}$ has the power series expansion
     \begin{gather}
      \label{eq:E_m}
      E_\mathfrak{m}(z)=\sum_{n=0}^{\infty}\frac{z^n}{\mathfrak{m}(n)} \quad  \textrm{for} \quad z\in\CC.
     \end{gather}
     In this case, the integral representation for the reciprocal moment function is given by
     \begin{gather}
      \label{eq:1/m}
     \frac{1}{\mathfrak{m}(u)}=\frac{1}{2\pi i}\int_{\gamma}E_\mathfrak{m}(w)w^{-u-1}dw
     \end{gather}
     with $\gamma$ as in Hankel's formula of the reciprocal Gamma function \cite[p.~228]{B2}. 
     \item[4.] Additionally we assume that $\mathfrak{m}(u)$ satisfies the normalization condition $\mathfrak{m}(0)=1$.
   \end{enumerate}
   \end{Def}
   
    Observe that in case $k\leq 1/2$ the set $S_{\pi}(2\pi-\pi/k)$ is not defined,
    so the second property in Definition \ref{df:moment} can not be satisfied. It means that we
    must define the kernel functions of order $k\leq 1/2$ and the corresponding moment functions
    in another way.
    
    \begin{Def}[see {\cite[Section 5.6]{B2}}]
     \label{df:small}
     A function $e_\mathfrak{m}$ is called \emph{a kernel function of order $k>0$} if we
     can find a pair of kernel functions $e_{\widetilde{\mathfrak{m}}}$ and $E_{\widetilde{\mathfrak{m}}}$ of
     order $pk>1/2$ (for some $p\in\NN$) so that
     \begin{gather*}
      e_\mathfrak{m}(z)=e_{\widetilde{\mathfrak{m}}}(z^{1/p})/p \quad \textrm{for} \quad z\in S_0(\pi/k).
     \end{gather*}
     For a given kernel function $e_\mathfrak{m}$ of order $k>0$ we define the corresponding 
     \emph{moment function $\mathfrak{m}$ of order $1/k>0$} by (\ref{eq:e_m}) and
     the \emph{kernel function $E_{\mathfrak{m}}$ of order $k>0$} by (\ref{eq:E_m}).
    \end{Def}
    
    \begin{Rem}
     \label{re:m_tilde}
     Observe that by Definitions \ref{df:moment} and \ref{df:small} we have
     \begin{eqnarray*}
      \mathfrak{m}(u)=\widetilde{\mathfrak{m}}(pu) & \textrm{and} &
      E_\mathfrak{m}(z)=\sum_{j=0}^{\infty}\frac{z^j}{\mathfrak{m}(j)}=\sum_{j=0}^{\infty}\frac{z^j}{\widetilde{\mathfrak{m}}(jp)}.
     \end{eqnarray*}
    \end{Rem}

    \begin{Rem}
    \label{re:general}
     By the general method of summability (see \cite[Section 6.5 and Theorem 38]{B2}), in the definition of $k$-summability (Definition \ref{df:summable}) one can replace the sequence $\Gamma_{1/k}=(\Gamma(1+n/k))_{n\geq 0}$ by any sequence $\mathfrak{m}=(\mathfrak{m}(n))_{n\geq 0}$,
     where $\mathfrak{m}(u)$ is a moment function of order $1/k$.
    \end{Rem}
    
    \begin{Rem}
    \label{re:moment_functions}
    By \cite[Theorems 31 and 32]{B2}, if $\mathfrak{m}_1(u)$ and $\mathfrak{m}_2(u)$ are moment functions of positive orders $1/k_1$ and $1/k_2$ respectively, then
    \begin{enumerate}
     \item $\mathfrak{m}(u)=\mathfrak{m}_1(u)\mathfrak{m}_2(u)$ is a moment function of order $1/k_1+1/k_2$,
     \item $\mathfrak{m}(u)=\mathfrak{m}_1(u)/\mathfrak{m}_2(u)$ is a moment function of order $1/k_1 - 1/k_2$ under condition that $1/k_1 > 1/k_2$.
    \end{enumerate}
    \end{Rem}
    
    Using the above remark we may extend the definition of moment functions to real order.
    \begin{Def}[{see \cite[Definition 4]{Mic8}}]
     \label{df:moment_general}
     We say that $\mathfrak{m}(u)$ is a \emph{moment function of order $1/k<0$} if $1/\mathfrak{m}(u)$ is a moment function of order $-1/k>0$.
     \par
     Moreover, $\mathfrak{m}(u)$ is called a \emph{moment function of order $0$} if there exist moment functions $\mathfrak{m}_1(u)$ and $\mathfrak{m}_2(u)$ of the same order $1/k>0$ such that $\mathfrak{m}(u)=\mathfrak{m}_1(u)/\mathfrak{m}_2(u)$.
    \end{Def}
    
    \begin{Rem}
     Observe that by Definitions \ref{df:moment}, \ref{df:small} and \ref{df:moment_general} any moment function $\mathfrak{m}(u)$ of order $s\in\RR$ satisfies conditions
     \begin{itemize}
      \item $\mathfrak{m}(u)>0$ for every $u\geq 0$,
      \item $\mathfrak{m}(0)=1$.
     \end{itemize}
     Hence the set $\mathfrak{M}$ given by (\ref{eq:M_got}) is a subset of the set $\mathcal{M}$ defined by (\ref{eq:M_cal}). We call $(\mathfrak{m}(n))_{n\geq  0}$ a \emph{sequence of moments inherited from $\mathfrak{m}(u)$}, if $\mathfrak{m}(u)$ is a moment function.
    \end{Rem}

    \begin{Rem}
     \label{re:group}
     By Remark \ref{re:moment_functions} and Definition \ref{df:moment_general}, the set of moment functions forms a group with a group operation given by the multiplication.
    \end{Rem}

    \begin{Rem}
     \label{rem:7} 
    By \cite[Section 5.5]{B2} and Definition \ref{df:moment_general}, if $\mathfrak{m}(u)$ is a moment function of order $s\in\RR$ then we see from \eqref{eq:e_m} and \eqref{eq:1/m} that $(\mathfrak{m}(n))_{n\geq 0}$ is a sequence of the same order $s$. 
     \end{Rem}

\section{Sequences preserving Gevrey order and summability} \label{sec:6} 
In this section we introduce and discuss new concepts of sequences, which preserve Gevrey order or summability. In particular we find the characterisation of such types of sequences.    
 \begin{Def}
 \label{df:preserve_g}
  We say that a sequence $m=(m(n))_{n\geq 0}$ \emph{preserves Gevrey order} if for any $s\in\RR$ and any $\hat{u}\in\EE[[t]]$ the following equivalence holds:
  \begin{equation}
  \label{eq:preserve_g}
  \hat{u}\in\EE[[t]]_s\quad\text{if and only if}\quad
  \Bo_{m,t}\hat{u}\in\EE[[t]]_{s}.
  \end{equation}
\end{Def}

In the next proposition we give a simple characterisation of sequences preserving Gevrey order. 
\begin{Prop}
\label{prop:preserve}
 A sequence $m=(m(n))_{n\geq 0}$ preserves Gevrey order if and only if $m$ is a sequence of order zero.
\end{Prop}
\begin{proof}
($\Rightarrow$) Applying Definition \ref{df:preserve_g} for $s=0$ and $\hat{u}(t)=\sum_{n=0}^{\infty}m(n)t^n$, and observing that $\Bo_{m,t}\hat{u}(t)=\sum_{n=0}^{\infty}t^n\in\CC[[t]]_0$, we conclude that also 
$\hat{u}(t)\in\CC[[t]]_0$. It means by Definition \ref{df:gevrey} that there exist $B,C<\infty$ such that
\begin{equation}
 \label{eq:est_mn}
 |m(n)|\leq BC^n\quad\text{for}\quad n\in\NN_0.
\end{equation}
Since $m(0)=1$ we may take $B=1$ in (\ref{eq:est_mn}) for sufficiently large $C<\infty$. In this way we get the right-hand side estimation in (\ref{eq:sequence_order}) with $A=C$.

Similarly, applying Definition \ref{df:preserve_g} for $s=0$ and $\hat{u}(t)=\sum_{n=0}^{\infty}t^n$ and observing that $\hat{u}(t)\in\CC[[t]]_0$, we conclude that also 
$\Bo_{m,t}\hat{u}(t)=\sum_{n=0}^{\infty}\frac{1}{m(n)}t^n\in\CC[[t]]_0$. It means by Definition \ref{df:gevrey} that there exist $B,C<\infty$ such that
\begin{equation*}
 \Big|\frac{1}{m(n)}\Big|\leq BC^n\quad\text{for}\quad n\in\NN_0.
\end{equation*} 
Since $m(0)=1$ and $m(n)>0$ we may take $B=1$ for sufficiently large $C<\infty$, so we conclude that
\begin{equation*}
 m(n)\geq C^{-n}\quad\text{for}\quad n\in\NN_0,
\end{equation*}
which gives the left-hand side estimation in (\ref{eq:sequence_order}) with $a=C^{-1}$.

($\Leftarrow$) Take any $s\in\RR$ and any $\hat{u}(t)=\sum_{n=0}^{\infty}a_nt^n$. If $\hat{u}\in\EE[[t]]_s$ and $m$ is a sequence of order zero then by (\ref{eq:sequence_order}) and (\ref{eq:est_un}) there exist constants $a>0$ and $B,C<\infty$ such that
$$
\Big\|\frac{a_n}{m(n)}\Big\|_{\EE}=m(n)^{-1}\|a_n\|_{\EE}\leq B(C/a)^n(n!)^s\quad\text{for}\quad n\in\NN_0.
$$
Hence also $\Bo_{m,t}\hat{u}(t)=\sum_{n=0}^{\infty}\frac{a_n}{m(n)}t^n\in\EE[[t]]_s$.

To prove the second implication in (\ref{eq:preserve_g}) we assume that $\Bo_{m,t}\hat{u}(t)=\sum_{n=0}^{\infty}\frac{a_n}{m(n)}t^n\in\EE[[t]]_s$. By (\ref{eq:sequence_order}) and (\ref{eq:est_un}) there exist constants $A,B,C<\infty$ such that
$$
\|a_n\|_{\EE}=m(n)\Big\|\frac{a_n}{m(n)}\Big\|_{\EE}\leq B(AC)^n(n!)^s\quad\text{for}\quad n\in\NN_0.
$$
It means that also $\hat{u}(t)=\sum_{n=0}^{\infty}a_nt^n\in\EE[[t]]_s$.
\end{proof}

In an analogous way we define sequences preserving summability.
 \begin{Def}
  We say that a sequence $m=(m(n))_{n\geq 0}$ \emph{preserves summability} if for any $k>0$, $d\in\RR$ and any $\hat{u}\in\EE[[t]]$ the following equivalence holds:
  \begin{equation*}
   \hat{u}\in\EE\{t\}_{k,d}\quad\text{if and only if}\quad
   \Bo_{m,t}\hat{u}\in\EE\{t\}_{k,d}.
  \end{equation*}
\end{Def}

\begin{Rem}
\label{re:inclusion}
Since for every $k>0$ and $d\in\RR$ we have $\EE\{t\}_{k,d}\subset \EE[[t]]_{1/k}$, we see that  
 if a sequence $m=(m(n))_{n\geq 0}$ preserves summability then this sequence $m$ preserves also Gevrey order (i.e. $m$ is a sequence of order $0$).
\end{Rem}

\begin{Rem}
 Directly by the definition of multisummability (Definition \ref{df:multisummable}) we conclude that if a sequence $m=(m(n))_{n\geq 0}$ preserves summability then it also preserves multisummability. It means that for every $\mathbf{k}=(k_1,\dots,k_N)$ with $k_1>\cdots>k_N>0$, for every admissible multidirection $\mathbf{d}=(d_1,\dots,d_N)\in\RR^N$ and for every $\hat{u}\in\EE[[t]]$ the following equivalence holds:
  \begin{equation*}
   \hat{u}\in\EE\{t\}_{\mathbf{k},\mathbf{d}}\quad\text{if and only if}\quad
   \Bo_{m,t}\hat{u}\in\EE\{t\}_{\mathbf{k},\mathbf{d}}.
  \end{equation*}
\end{Rem}

\begin{Ex}
\label{ex:ex_2}
 Examples of sequences preserving summability:
 \begin{enumerate}
  \item If $a>0$ and $\mathbf{a}:=(a^n)_{n\geq 0}$ then the sequence $\mathbf{a}$ preserves summability. In particular the sequence $\mathbf{1}=(1)_{n\geq 0}$ preserves summability in a trivial way.
  \item By Balser's theory of general summability \cite[Section 6.5 and Theorem 38]{B2} for any moment function $\mathfrak{m}(u)$ of order zero, the sequence $(\mathfrak{m}(n))_{n\geq 0}$ preserves summability (see also Remark \ref{re:general}).
 \end{enumerate}
\end{Ex}

\begin{Ex}
\label{ex:ex_3}
 Not every sequence of order $0$ preserves summability. Let
 \begin{equation}
 \label{eq:ex_m}
  m(n):=
    \left\{
    \begin{array}{ll}
     1& n\ \text{is even}\\
     2^{-1}& n\ \text{is odd}
    \end{array}
    \right..
 \end{equation}
The series $\hat{x}(t)=\sum_{n=0}^{\infty}n!t^n$ is $1$-summable in any direction $d\neq 0 \mod 2\pi$, because for $m_1(n)=n!$ and for any $d\neq 0 \mod 2\pi$
$$
\Bo_{m_1,t}\hat{x}(t)=\sum_{n=0}^{\infty}t^n=\frac{1}{1-t}\in\Oo^1(\hat{S}_d).
$$

On the other hand the series 
$$
\hat{y}(t)=\Bo_{m,t}\hat{x}(t)=\sum_{n=0}^{\infty}\frac{n!}{m(n)}t^n=\sum_{k=0}^{\infty}(2k)!t^{2k}+\sum_{k=0}^{\infty}2(2k+1)!t^{2k+1}
$$
is $1$-summable only for directions $d\neq 0 \mod \pi$, because the function
$$
\Bo_{m_1,t}\hat{y}(t)=\sum_{k=0}^{\infty}t^{2k}+\sum_{k=0}^{\infty}2 t^{2k+1}=\frac{1}{1-t^{2}}+\frac{2t}{1-t^{2}}=\frac{1+2t}{1-t^2}\in\Oo^1(\hat{S}_d),\ d\neq 0 \mod \pi
$$
has a simple pole not only at $t=1$, but also at $t=-1$.

Hence $\hat{x}(t)\in\CC\{t\}_{1,\pi}$, but $\hat{y}(t)=\Bo_{m,t}\hat{x}(t)\not\in\CC\{t\}_{1,\pi}$.
\end{Ex}
\begin{Rem}
Observe that by Example \ref{ex:ex_3} the sequence of order zero $(m(n))_{n\geq 0}$ defined by (\ref{eq:ex_m}) satisifies $m(n)>0$ for $n\geq 0$ and $m(0)=1$, but it does not preserve summability. On the other hand, by Example \ref{ex:ex_2} every sequence of moments inherited from some moment function $\mathfrak{m}(u)$ of order zero preserves summability. Hence
the sequence $(m(n))_{n\geq 0}$ defined by (\ref{eq:ex_m}) is not a sequence of moments inherited from any moment function $\mathfrak{m}(u)$ of order zero. 
\end{Rem}

\begin{Rem}
\label{re:groups}
 The set of sequences preserving Gevrey order (resp. summability) forms a group with a group operation given by the multiplication. If $m_1=(m_1(n))_{n\geq0}$ and $m_2=(m_2(n))_{n\geq 0}$ preserve Gevrey order (resp. summability) then also their product $m=m_1\cdot m_2$ (i.e. $m=(m(n))_{n\geq 0}$, where $m(n)=m_1(n)\cdot m_2(n)$ for any $n\in\NN_0$) preserves Gevrey order (resp. summability). Observe also, that the identity element $\mathbf{1}=(1)_{n\geq 0}$ and the inverse element $m^{-1}=(m(n)^{-1})_{n\geq 0}$ to $m=(m(n))_{n\geq 0}$ preserve Gevrey order (resp. summability).
 
 Note that by Remark \ref{re:inclusion} and Example \ref{ex:ex_3} the group of sequences preserving summability is a proper subgroup of the group of sequences preserving Gevrey order.
 
 Moreover, by Remark \ref{re:group}  the set
 \begin{equation}
 \label{eq:subgroup}
 \mathfrak{M}_0=\Big\{(\mathfrak{m}(n))_{n\geq 0}\colon \mathfrak{m}(u)\ \text{is a moment function of order zero}\Big\}
 \end{equation}
 with the multiplication forms a subgroup of these two groups mentioned above.
\end{Rem}

\section{Moment differentiations and moment operators of order $0$} \label{sec:7} 
In this section we extend the notion of $m$-moment differentiation introduced by Balser and Yoshino \cite{B-Y} to any sequence $m=(m(n))_{n\geq 0}$ of positive numbers with $m(0)=1$. We introduce 
an $m$-moment differentiation of order $r$ ($r\in\RR$) and we focus our special attention on the case $r=0$. 
\begin{Def}
 For a given sequence $m=(m(n))_{n\geq 0}$ of positive numbers with $m(0)=1$, an operator $\partial_{m,t}\colon\EE[[t]] \to \EE[[t]]$ defined by
 \begin{equation*}
  \partial_{m,t}\big(\sum_{n=0}^{\infty}\frac{u_n}{m(n)}t^n\big):=\sum_{n=0}^{\infty}\frac{u_{n+1}}{m(n)}t^n
 \end{equation*}
is called an \emph{$m$-moment differentiation}.

If additionally $m$ is a sequence of order $r$ for fixed $r\in\RR$ then $\partial_{m,t}$ is called an \emph{$m$-moment differentiation of order $r$} or an \emph{operator of order $r$} for short.
\end{Def}

\begin{Rem}
 \label{re:diff}
 Notice that for any formal power series $\hat{u}(t)=\sum_{n=0}^{\infty}a_n t^n\in\EE[[t]]$ the operator $\partial_{m,t}\colon\EE[[t]] \to \EE[[t]]$ can be equivalently defined as
 \begin{equation*}
  \partial_{m,t}\hat{u}(t)=\partial_{m,t}\big(\sum_{n=0}^{\infty}a_nt^n\big)=\sum_{n=0}^{\infty}\frac{m(n+1)}{m(n)}a_{n+1}t^n.
 \end{equation*}

\end{Rem}

\begin{Rem}
 Observe that in the most important case $m=(n!)_{n\geq 0}$, the operator $\partial_{m,t}$ is the $m$-moment differentiation of order $1$, which coincides with the usual differentiation $\partial_t$.
\end{Rem}

By the direct calculation we get
\begin{Prop}
\label{prop:commutate}
 Let $m_1=(m_1(n))_{n\geq 0}$ and $m_2=(m_2(n))_{n\geq 0}$ be sequences of positive numbers. Then the operators $\Bo_{m_1,t},\partial_{m_2,t}\colon\EE[[t]]\to\EE[[t]]$ commute in a such way that
 \begin{equation*}
  \Bo_{m_1,t}\partial_{m_2,t}=\partial_{m_1m_2,t}\Bo_{m_1,t}.
 \end{equation*}
\end{Prop}

\begin{Prop}[The moment Taylor formula]
\label{prop:taylor}
 Let $\hat{\varphi}\in\EE[[t]]$ and $m=(m(n))_{n\geq 0}$ be a sequence of positive numbers with $m(0)=1$. Then
 \begin{equation}
 \label{eq:taylor}
  \hat{\varphi}(t)=\sum_{n=0}^{\infty}\frac{\partial^n_{m,t}\hat{\varphi}(0)}{m(n)}t^n.
 \end{equation}
\end{Prop}
\begin{proof}
 Let $\hat{\varphi}(t)=\sum_{k=0}^{\infty}a_kt^k\in\EE[[t]]$. Using the definition of $m$-moment differentiation (see also Remark \ref{re:diff}) we conclude that for any $n\in\NN_0$
 \begin{equation*}
  \partial^n_{m,t}\hat{\varphi}(t)=\partial^n_{m,t}\big(\sum_{k=0}^{\infty}a_kt^k\big)=\sum_{k=0}^{\infty}\frac{m(k+n)}{m(k)}a_{k+n}t^k.
 \end{equation*}
It means that $\partial^n_{m,t}\hat{\varphi}(0)=\frac{m(n)}{m(0)}a_n=m(n)a_n$. Hence we get the conclusion (\ref{eq:taylor}).
\end{proof}

\begin{Rem}
\label{re:u}
In the paper we will also use Proposition \ref{prop:taylor} for $\hat{u}(t,z)\in\Oo(D,\EE)[[t]]$. In this case $\hat{u}(\cdot,z)\in\EE[[t]]$ for any fixed $z\in D$ and (\ref{eq:taylor}) has the form
\begin{equation*}
  \hat{u}(t,z)=\sum_{n=0}^{\infty}\frac{\partial^n_{m,t}\hat{u}(0,z)}{m(n)}t^n.
 \end{equation*}
\end{Rem}

\begin{Rem}
\label{re:l}
 Since the left-hand side of of (\ref{eq:taylor}) does not depend on a sequence $m$, by Proposition \ref{prop:taylor} we conclude that for any sequences $m_1=(m_1(n))_{n\geq 1}$ and $m_2=(m_2(n))_{n\geq 1}$
 \begin{equation*}
  \frac{\partial^n_{m_1,t}\hat{\varphi}(0)}{m_1(n)}=\frac{\partial^n_{m_2,t}\hat{\varphi}(0)}{m_2(n)}
  \quad\textrm{for every}\quad n\in\NN_0.
 \end{equation*}
 Especially, when $m_1=m$ and $m_2(n)=n!$, it holds that $\displaystyle \frac{\partial_{m,t}^n\hat{\varphi}(0)}{m(n)}=\frac{\partial_t^n\hat{\varphi}(0)}{n!}$ for every $n\in\NN_0$.
\end{Rem}

\begin{Prop}
\label{prop:aux0}
Let ${m}_i(n)$ be a sequence of positive numbers with $m_i(0)=1$ for $i=1,2$. Then the Cauchy problem
\begin{equation*}
\label{eq:aux0}
    \left\{
    \begin{array}{l}
     (\partial_{{m_1},t}-\partial_{m_2,z})u=0\\
     u(0,z)=\hat{\varphi}(z)\in\EE[[z]]
    \end{array}
    \right.
 \end{equation*}
 has a unique formal power series solution
 $$
 \hat{u}(t,z)=\sum_{n\geq 0}\frac{\partial_{m_2,z}^n\hat{\varphi}(z)}{m_1(n)}t^n.
 $$
\end{Prop}
\begin{proof}
By substituting $\hat{u}(t,z)=\sum_{n\geq0}\frac{u_n(z)}{m_1(n)}t^n$, we have $u_{n+1}(z)=\partial_{m_2,z}u_n(z)$ and $u_0(z)=\hat{\varphi}(z)$. Therefore we have $u_n(z)=\partial_{m_2,z}^n\hat{\varphi}(z)$.
\end{proof}

\begin{Ex}
 Examples of operators of order $0$:
 \begin{enumerate}
  \item If $\mathbf{1}=(1)_{n\geq 0}$ then
  \begin{equation*}
   \partial_{\mathbf{1},t}\hat{u}(t)=\frac{\hat{u}(t)-\hat{u}(0)}{t}\quad\text{for}\quad \hat{u}(t)\in\EE[[t]].
  \end{equation*} 
  More generally, for every $n\in\NN$ we get
  \begin{equation*}
   \partial^n_{\mathbf{1},t}\hat{u}(t)=\frac{\hat{u}(t)-\sum_{k=0}^{n-1}\frac{\partial^k_t\hat{u}(0)}{k!}t^k}{t^n}\quad\text{for}\quad \hat{u}(t)\in\EE[[t]].
  \end{equation*}
  Hence we may write the usual Taylor's theorem as
  \begin{equation*}
   \hat{u}(t)=\sum_{k=0}^{n-1}\frac{\partial^k_t\hat{u}(0)}{k!}t^k+R_n(t),
  \end{equation*}
where the reminder term $R_n(t)$ of the Taylor polynomial is given by $R_n(t)=t^n\partial^n_{\mathbf{1},t}\hat{u}(t)$.
  \item If $a>0$ and $\mathbf{a}=(a^n)_{n\geq 0}$ then 
  \begin{equation*}
   \partial_{\mathbf{a},t}\hat{u}(t)=\frac{a(\hat{u}(t)-\hat{u}(0))}{t}=a\partial_{\mathbf{1},t}\hat{u}(t)\quad\text{for}\quad \hat{u}(t)\in\EE[[t]].
  \end{equation*}
 \item By Remark \ref{rem:7} if $\mathfrak{m}(u)$ is a moment function of order $0$ then $\mathfrak{m}=(\mathfrak{m}(n))_{n\geq 0}$ is a sequence of order $0$. Hence $\partial_{\mathfrak{m},t}$ is an operator of order $0$.
  \item Let $m=([n]_q!)_{n\geq 0}$ for $q\in[0,1)$. Observe that
  $D_{q,t}t^n=[n]_qt^{n-1}$, hence in this case $\partial_{m,t}$ coincides with the $q$-difference operator $D_{q,t}$, i.e.
  \begin{equation*}
   \partial_{m,t}\hat{u}(t)=D_{q,t}\hat{u}(t)=\frac{\hat{u}(qt)-\hat{u}(t)}{qt-t}\quad\text{for}\quad \hat{u}(t)\in\EE[[t]].
  \end{equation*}
  Since $q\in [0,1)$ we get $1\leq [n]_q\leq \frac{1}{1-q}$ for every $n\in\NN_0$ and we conclude that
  \begin{equation}
   \label{eq:est_nq!}
   1\leq [n]_q!\leq \big(\frac{1}{1-q}\big)^n\quad\textrm{for every}\ n\in\NN_0.
  \end{equation}
  Therefore the $q$-difference operator $D_{q,t}$ is the $m$-moment differentiation of order $0$.\\
  Observe also that in the special case $q=0$ we get
  $$
  D_{0,t}\hat{u}(t)=\partial_{\mathbf{1},t}\hat{u}(t)=\frac{\hat{u}(t)-\hat{u}(0)}{t}.
  $$
 \end{enumerate}
 \end{Ex}

\section{The characterisation of sequences preserving summability} \label{sec:7+1} 
Proposition \ref{prop:preserve} gives a full characterisation of sequences preserving Gevrey order. We show a similar characterisation for sequences preserving summability. 

 \begin{Thm}
 \label{th:sequence_summability}
  A sequence $m=(m(n))_{n\geq 0}$ preserves summability if and only if for every $k>0$ and for every $\theta\neq 0\mod 2\pi$ there exists a disc-sector $\hat{S}_{\theta}$ such that
  \begin{equation*}
  \Bo_{m,t}\big(\sum_{n=0}^{\infty}t^n\big)\in\Oo^k(\hat{S}_{\theta})\quad \textrm{and}\quad \Bo_{m^{-1},t}\big(\sum_{n=0}^{\infty}t^n\big)\in\Oo^k(\hat{S}_{\theta}).
  \end{equation*}
 \end{Thm}
 \begin{proof}
 \par
 ($\Rightarrow$)
 Take any $k>0$ and $\theta\neq 0 \mod 2\pi$. Let $\hat{u}(t):=\sum_{n=0}^{\infty}\Gamma(1+n/k)t^n$. Since 
 $$
 \Bo_{\Gamma_{1/k},t}\hat{u}(t)=\sum_{n=0}^{\infty}t^n=\frac{1}{1-t}\in\Oo^k(\hat{S}_{\theta}),
 $$
 we see that $\hat{u}$ is $k$-summable in a direction $\theta$. It means that also $\Bo_{m,t}\hat{u}(t)$ and $\Bo_{m^{-1},t}\hat{u}(t)$ are $k$-summable in a direction $\theta$ for any sequence $m$ preserving summability. Hence we conclude that 
 \begin{equation*}
  \Bo_{\Gamma_{1/k},t}\big(\Bo_{m,t}\hat{u}\big)=\Bo_{m,t}\big(\sum_{n=0}^{\infty}t^n\big)\in\Oo^k(\hat{S}_{\theta})\quad\textrm{and}\quad
  \Bo_{\Gamma_{1/k},t}\big(\Bo_{m^{-1},t}\hat{u}\big)=\Bo_{m^{-1},t}\big(\sum_{n=0}^{\infty}t^n\big)\in\Oo^k(\hat{S}_{\theta}).
 \end{equation*}
 
 ($\Leftarrow$)
 Take any $k>0$ and $d\in\RR$. Assume that $\hat{x}(t)=\sum_{n=0}^{\infty}x_nt^n\in\EE[[t]]$ is $k$-summable in a direction~$d$. It is sufficient to show that also $\Bo_{m,t}\hat{x}(t)$ and $\Bo_{m^{-1},t}\hat{x}(t)$ are $k$ summable in the same direction~$d$.
 
 Since $\hat{x}(t)\in\EE\{t\}_{k,d}$, we see that the function $\varphi(t):=\Bo_{\Gamma_{1/k},t}\hat{x}(t)$ belongs to the space $\Oo^k(\hat{S}_d,\EE)$. Let $\hat{u}(t,z)$ be a solution of the Cauchy problem
\begin{equation}
\label{eq:auxiliary}
    \left\{
    \begin{array}{l}
     (\partial_{\tilde{m},t}-\partial_{z})u=0\\
     u(0,z)=\varphi(z)\in\Oo^k(\hat{S}_d,\EE),
    \end{array}
    \right.
 \end{equation}
where $\tilde{m}=(m(n)n!)_{n\geq 0}$. Then since $\hat{u}(t,z)=\sum_{n\geq0}\frac{\partial_z^n\varphi(z)}{\tilde{m}(n)}t^n$ from Proposition \ref{prop:aux0}, we have
\begin{equation*}
 \hat{u}(t,0)=\sum_{n=0}^{\infty}\frac{\varphi^{(n)}(0)}{m(n)n!}t^n=\sum_{n=0}^{\infty}\frac{x_n}{\Gamma(1+n/k)m(n)}t^n=
 \Bo_{\Gamma_{1/k},t}\big(\Bo_{m,t}\hat{x}(t)\big),
\end{equation*}
which is convergent at $t=0$. 
To prove that $\Bo_{m,t}\hat{x}(t)$ is $k$-summable in the direction $d$, it is sufficient to show that $u(t,0)\in\Oo^k(\hat{S}_d,\EE)$. To this end observe that using the integral representation of $u$ we get
\begin{equation}
 \label{eq:u(t,0)}
 u(t,0)=\frac{1}{2\pi i}\oint_{|\zeta|=\rho}\frac{\varphi(\zeta)}{\zeta}\Big(\sum_{n=0}^{\infty}\frac{(t/\zeta)^n}{m(n)}\Big)\,d\zeta=\frac{1}{2\pi i}\oint_{|\zeta|=\rho}\frac{\varphi(\zeta)}{\zeta}\psi(t/\zeta)\,d\zeta
\end{equation}
for sufficiently small $\rho>0$, where the kernel $\psi$ is defined as
\begin{equation}
\label{eq:kernel}
\psi(t):=\sum_{n=0}^{\infty}\frac{t^n}{m(n)}=\Bo_{m,t}\big(\sum_{n=0}^{\infty}t^n\big).
\end{equation}
Observe that  $\psi\in\Oo^{k}(\hat{S}_{\theta})$ for every $\theta\neq 0\mod 2\pi$ by the assumption. In particular the power series $\psi$ has a positive radius of convergence $r>0$, i.e. $\psi\in\Oo(D_r)$.   
Since $\varphi\in\Oo(\hat{S}_d,\EE)$ and $\psi\in\Oo(\hat{S}_{\theta})$ for every $\theta\neq 0\mod 2\pi$ we may deform the path of integration in (\ref{eq:u(t,0)}) from $\zeta\in\partial D_{\rho}$ to $\zeta\in\Gamma(R):=\partial(\hat{S}_d\cap D_R)$.  
Exactly, whenever $d\not=0\mod 2\pi$, we can deform such a path of integration. When $d=0\mod 2\pi$, we can deform such a path of integration that $|\frac{t}{\zeta}|$ is bounded as $|t|\to\infty$ in $\hat{S}_d$. 
Taking $R\to\infty$ we see that $u(t,0)\in\Oo(\hat{S}_d,\EE)$ for some disc-sector $\hat{S}_d$.

To estimate $\|u(t,0)\|_{\EE}$ for $t\in\hat{S}_d$, $|t|\to\infty$, we split the contour $\Gamma(R)$ into 2 arcs $\Gamma_1(R):=\Gamma(R)\cap(\partial D_R)$ and $\Gamma_2(R):=\Gamma(R)\cap D_R$. Then we get
\begin{equation}
\label{eq:split}
u(t,0)=\frac{1}{2\pi i}\oint_{\Gamma(R)}\frac{\varphi(\zeta)}{\zeta}\psi(t/\zeta)\,d\zeta=\frac{1}{2\pi i}\int_{\Gamma_1(R)}\frac{\varphi(\zeta)}{\zeta}\psi(t/\zeta)\,d\zeta+\frac{1}{2\pi i}\int_{\Gamma_2(R)}\frac{\varphi(\zeta)}{\zeta}\psi(t/\zeta)\,d\zeta.
\end{equation}
We may
estimate $u(t,0)$ for $t\in\hat{S}_d$ with $|t|\to\infty$, as in \cite[Lemma 5]{Mic7} or \cite[Lemma 4]{Mic8}.

Namely, if $\zeta\in\Gamma_1(R)$ then $|\zeta|=R$ and $\zeta\in\hat{S}_d$. Taking $R=2|t|/r$, where the constant $r>0$ is the radius of convergence of $\psi$, we see that $R$ and $t$ both go to infinity together and that the function $t\mapsto\psi(t/\zeta)$ is bounded. Since moreover $\varphi\in\Oo^k(\hat{S}_d,\EE)$, we conclude that the first integral on the right-hand side of (\ref{eq:split}) has exponential growth of order $k$ as $|t|\to\infty$ in $\hat{S}_d$.

To estimate the second integral, observe that if $\zeta\in\Gamma_2(R)$ then $\arg\zeta\neq d\mod2\pi$. It means that the function $t\mapsto\psi(t/\zeta)$ has exponential growth of order $k$ as $|t|\to\infty$ in $\hat{S}_d$. Since moreover $\varphi\in\Oo^k(\hat{S}_d,\EE)$, in this case we also conclude that the second integral on the right-hand side of (\ref{eq:split}) has exponential growth of order $k$ as $|t|\to\infty$ in $\hat{S}_d$.

Hence the function $t\mapsto u(t,0)$ has also exponential growth of order $k$ as $|t|\to\infty$ in $\hat{S_d}$ and $\Bo_{m,t}\hat{x}(t)$ is $k$-summable in the direction $d$.

Replacing $m$ by $m^{-1}$ and repeating the above proof we conclude  that $\Bo_{m^{-1},t}\hat{x}(t)$ is also $k$-summable in the same direction $d$.
\end{proof}

\begin{Rem}
 In the special case when $m=(m(n))_{n\geq 0}=(\mathfrak{m}(n))_{n\geq 0}$ for some moment function $\mathfrak{m}(u)$ of order $0$, by Balser's moment method \cite{B2} we may express the kernel $\psi$ given by (\ref{eq:kernel}) as
 \begin{equation*}
  \psi(t)=\tilde{\Bo}_{m,t}\big(\frac{1}{1-t}\big)\quad\textrm{for some integral operator}\quad \tilde{\Bo}_{m,t}.
 \end{equation*}
Moreover, since by Definition \ref{df:moment_general} we may write $\mathfrak{m}(u)=\mathfrak{m}_1(u)/\mathfrak{m}_2(u)$ for some fixed moment functions $\mathfrak{m}_1(u)$ and $\mathfrak{m}_2(u)$
of the same order $1/k<2$, by \cite[Section 5.5]{B2} we conclude that
$\tilde{\Bo}_{m,t}=T^{-}_{\mathfrak{m}_1}\circ T^{+}_{\mathfrak{m}_2}$,
where $T^{-}_{\mathfrak{m}_1}$ and $T^{+}_{\mathfrak{m}_2}$ are two integral operators which are generalisations of Borel and Laplace operators and are  defined respectively by \cite[formula (5.14)]{B2} and \cite[formula (5.13)]{B2} in terms of the kernel functions $E_{\mathfrak{m}_1}$ and $e_{\mathfrak{m}_2}$. 
\end{Rem}

\begin{Rem}
 Observe, that we can formulate Proposition \ref{prop:preserve} in the similar way to Theorem \ref{th:sequence_summability}:
 
 \emph{
  A sequence $m=(m(n))_{n\geq 0}$ preserves Gevrey order if and only if there exists a disc $D$ such that
  \begin{equation*}
  \Bo_{m,t}\big(\sum_{n=0}^{\infty}t^n\big)\in\Oo(D)\quad \textrm{and}\quad \Bo_{m^{-1},t}\big(\sum_{n=0}^{\infty}t^n\big)\in\Oo(D).
  \end{equation*}
  }
\end{Rem}

\begin{Rem}
 The characterisation of sequences preserving summability given in Theorem \ref{th:sequence_summability} is similar in spirit to \cite[Lemma 6]{B0}.
\end{Rem}

\begin{Rem}
 In the proof of Theorem \ref{th:sequence_summability} we apply a new idea of use of auxiliary Cauchy problem for the simple moment equation (\ref{eq:auxiliary}), where we  study the relation between the initial data $\varphi(z)=u(0,z)$ and the boundary data $t\mapsto u(t,0)$.
 This auxiliary Cauchy problem and its solution gives us a new tool which transforms the power series $z\mapsto u(0,z)$ onto the power series $t\mapsto u(t,0)$.
 The same tool but for other simple moment equations is  used in the next section (see Lemmas \ref{lem:1}, \ref{lem:2} and \ref{lem:3} below) to prove Theorem \ref{th:main}.    
 \end{Rem}

\section{The main result} \label{sec:8} 
By (\ref{eq:est_nq!}) the sequence  $([n]_q!)_{n\geq 0}$ preserves Gevrey order. In this section we prove the main result of the paper, which says that this sequence 
also preserves summability.

To this end we need a few lemmas. The first one shows that the solution of the moment equation (\ref{eq:v}) has the same boundary and initial condition.
\begin{Lem}
\label{lem:1}
 Let $k>0$, $d\in\RR$ and $\mathfrak{m}(u)$ be a moment function of order $1/k$.
 Let $\hat{v}(t,z)\in\Oo(D,\EE)[[t]]$ be a formal solution of the Cauchy problem
 \begin{equation}
 \label{eq:v}
    \left\{
    \begin{array}{l}
     (\partial_{\mathfrak{m},t}-\partial_{\mathfrak{m},z})v=0\\
     v(0,z)=\varphi(z)\in\Oo(D,\EE).
    \end{array}
    \right..
 \end{equation}
 Then $v(t,z)\in\Oo(D^2,\EE)$ and $\psi(t)=\varphi(t)$, where $\psi(t):=v(t,0)$.
\end{Lem}
\begin{proof}
From Proposition \ref{prop:aux0}, the formal power series solution of (\ref{eq:v}) is given by
\begin{equation}
\label{eq:v_series}
 \hat{v}(t,z)=\sum_{n=0}^{\infty}\frac{\partial^n_{\mathfrak{m},z}\varphi(z)}{\mathfrak{m}(n)}t^n.
\end{equation}
By \cite[Lemma 1]{Mic7} there exists $r>0$ and $A,B<\infty$ such that
\begin{equation}
\label{eq:phi}
 \sup_{|z|<r} \|\partial^n_{\mathfrak{m},z}\varphi(z)\|_{\EE}\leq A B^n \mathfrak{m}(n)\quad\textrm{for every}\quad n\in\NN_0.
\end{equation}
Hence the formal power series solution $\hat{v}(t,z)$ given by (\ref{eq:v_series}) is convergent for $|z|<r$ and $|t|<B^{-1}$, so $v(t,z)\in\Oo(D^2,\EE)$.

Evaluating $v(t,0)$ in (\ref{eq:v_series}) and using the moment Taylor formula for $\varphi(z)$ (Proposition \ref{prop:taylor} and Remark \ref{re:u}) we conclude that
\begin{equation*}
 \psi(t)=v(t,0)=\sum_{n=0}^{\infty}\frac{\partial^n_{\mathfrak{m},z}\varphi(0)}{\mathfrak{m}(n)}t^n=\varphi(z)|_{z=t}=\varphi(t).
\end{equation*}
\end{proof}

Similarly, the next two lemmas show that the initial and boundary data of the solution of the moment equation (\ref{eq:v_tilde_1}) have the same analytic continuation properties. These results are similar in spirit to \cite[Theorem 3.1]{I-A}.

First, following the proof of sufficiency in \cite[Theorem 3.1]{I-A} we will show
\begin{Lem}
\label{lem:2}
 Let $k>0$, $d\in\RR$, $\mathfrak{m}(u)$ be a moment function of order $1/k$ and $\tilde{m}:=(\mathfrak{m}(n)[n]_q!)_{n\geq 0}$ for $q\in(0,1)$.
 Let $\hat{u}(t,z)\in\Oo(D,\EE)[[t]]$ be a formal solution of the initial value problem
 \begin{equation}
 \label{eq:v_tilde_1}
    \left\{
    \begin{array}{l}
     (\partial_{\tilde{m},t}-\partial_{\mathfrak{m},z})u=0\\
     u(0,z)=\varphi(z)\in\Oo(D,\EE)
    \end{array}
    \right..
 \end{equation}
 Then $u(t,z)\in\Oo(D^2,\EE)$. If additionally
 $\varphi(z)\in\Oo^k(\hat{S}_d,\EE)$ then also
 $\tilde{\psi}(t)\in\Oo^k(\hat{S}_d,\EE)$, where $\tilde{\psi}(t):=u(t,0)$.
\end{Lem}
\begin{proof}
From Proposition \ref{prop:aux0}, the formal power series solution of (\ref{eq:v_tilde_1}) is given by
\begin{equation}
\label{eq:v_tilde}
 \hat{u}(t,z)=\sum_{n=0}^{\infty}\frac{\partial^n_{\mathfrak{m},z}\varphi(z)}{\tilde{m}(n)}t^n.
\end{equation}
By (\ref{eq:phi}) and (\ref{eq:est_nq!}) the formal power series solution $\hat{u}(t,z)$ given by (\ref{eq:v_tilde}) is convergent for $|z|<r$ and $|t|<B^{-1}$, so $u(t,z)\in\Oo(D^2,\EE)$.

Using the moment Taylor formula (Proposition \ref{prop:taylor}) and Taylor formula for $\varphi(z)\in\Oo(D,\EE)$ (see also Remark \ref{re:l}) we conclude that 
\begin{equation*}
 \frac{\partial^n_{\mathfrak{m},z}\varphi(0)}{\mathfrak{m}(n)}=\frac{\varphi^{(n)}(0)}{n!}\quad\textrm{for every}\quad n\in\NN_0.
\end{equation*}
Hence by (\ref{eq:v_tilde})
\begin{equation*}
 u(t,0)=\sum_{n=0}^{\infty}\frac{\partial^n_{\mathfrak{m},z}\varphi(0)}{\tilde{m}(n)}t^n=\sum_{n=0}^{\infty}\frac{\partial^n_{\mathfrak{m},z}\varphi(0)}{\mathfrak{m}(n)[n]_q!}t^n=
 \sum_{n=0}^{\infty}\frac{\varphi^{(n)}(0)}{n![n]_q!}t^n.
\end{equation*}

We will follow the proof of sufficiency in \cite[Theorem 3.1]{I-A} with $\kappa=\nu=1$ and $x=0$. Since $[n]_q!=\frac{(q;q)_n}{(1-q)^n}$, by the Cauchy integral formula we get
\begin{equation*}
 \tilde{\psi}(t)=u(t,0)=\frac{1}{2\pi i}\oint_{|\zeta|=\rho}\frac{\varphi(\zeta)}{\zeta}\sum_{n=0}^{\infty}\frac{1}{(q;q)_n}\Big(\frac{(1-q)t}{\zeta}\Big)^n\,d\zeta
 \end{equation*}
 for sufficiently small $|t|$.
 Moreover, by (\ref{eq:q_binomial}) and the $q$-binomial theorem (Proposition \ref{prop:1})
 \begin{equation*}
 \sum_{n=0}^{\infty}\frac{1}{(q;q)_n}\Big(\frac{(1-q)t}{\zeta}\Big)^n=
 {}_{1}\phi_0\left(\begin{array}{c}
                       0\\
                       -
                     \end{array}
;q,\frac{(1-q)t}{\zeta}\right)=\frac{1}{\Big(\frac{(1-q)t}{\zeta};q\Big)_{\infty}}=
\prod_{n=0}^{\infty}\frac{\zeta}{\zeta-(1-q)tq^n}.
 \end{equation*}
 Observe that for fixed $t\neq 0$ the function 
 \begin{equation*}
 \zeta\longmapsto\prod_{n=0}^{\infty}\frac{\zeta}{\zeta-(1-q)tq^n}
 \end{equation*}
 is meromorphic on $\CC$ with simple poles at
\begin{equation*}
 \zeta=\zeta_n(t):=(1-q)tq^n\quad\text{for}\quad n\in\NN_0.
\end{equation*}
Hence, by the residue theorem we get
\begin{equation*}
 \tilde{\psi}(t)=\sum_{n=0}^{\infty}\varphi((1-q)tq^n)\mathop{\mathrm{Res}}_{\zeta=\zeta_n(t)}\frac{1}{\big(\frac{(1-q)t}{\zeta};q\big)_{\infty}}\frac{1}{\zeta}
 =\sum_{n=0}^{\infty}\varphi((1-q)tq^n)\frac{(-1)^n q^{\frac{n(n+1)}{2}}}{(q;q)_n(q;q)_{\infty}}.
\end{equation*}
Since $\varphi(z)\in\Oo^k(\EE,\hat{S}_d)$, there exist $A,B<\infty$ such that $\|\varphi(z)\|_{\EE}\leq Ae^{B|z|^k}$ for every $z\in\hat{S}_d$. Hence
\begin{equation*}
 \|\tilde{\psi}(t)\|_{\EE}\leq \frac{A}{(q;q)_{\infty}}e^{B(1-q)^k|t|^k}\sum_{n=0}^{\infty}\frac{q^{\frac{n(n+1)}{2}}}{(q;q)_n}\leq \frac{A}{(q;q)_{\infty}}e^{\tilde{B}|t|^k}\sum_{n=0}^{\infty}\Big(\frac{q^{\frac{n+1}{2}}}{1-q}\Big)^n\leq \tilde{A}e^{\tilde{B}|t|^k}
\end{equation*}
for some positive constants $\tilde{A},\tilde{B}<\infty$ and for every $t\in \hat{S}_d$. It means that $\tilde{\psi}(t)\in\Oo^k(\hat{S}_d,\EE)$.  
\end{proof}

Next, following the proof of necessity in \cite[Theorem 3.1]{I-A} we will prove
\begin{Lem}
\label{lem:3}
Let $k>0$, $d\in\RR$, $\mathfrak{m}(u)$ be a moment function of order $1/k$ and $\tilde{m}:=(\mathfrak{m}(n)[n]_q!)_{n\geq 0}$ for $q\in(0,1)$.
 Let $\hat{u}(t,z)\in\Oo(D,\EE)[[z]]$ be a formal solution of the boundary value problem
 \begin{equation}
  \label{eq:v_tilde_2}
    \left\{
    \begin{array}{l}
     (\partial_{\tilde{m},t}-\partial_{\mathfrak{m},z})u=0\\
     u(t,0)=\tilde{\psi}(t)\in\Oo(D,\EE)
    \end{array}
    \right.
 \end{equation}
 Then $u(t,z)\in\Oo(D^2,\EE)$. If additionally
 $\tilde{\psi}(t)\in\Oo^k(\hat{S}_d,\EE)$ then also
 $\varphi(z)\in\Oo^k(\hat{S}_d,\EE)$, where $\varphi(z):=u(0,z)$.
\end{Lem}
\begin{proof}
By Proposition \ref{prop:aux0} with replaced variables, the formal power series solution of (\ref{eq:v_tilde_2}), belonging to $\Oo(D,\EE)[[z]]$,  is given by
\begin{equation}
\label{eq:v_tilde2}
 \hat{u}(t,z)=\sum_{n=0}^{\infty}\frac{\partial^n_{\tilde{m},t}\tilde{\psi}(t)}{\mathfrak{m}(n)}z^n.
\end{equation}
By \cite[Proposition 1]{L-Mi-Su} and (\ref{eq:est_nq!})
 there exists $r>0$ and $A,B<\infty$ such that
\begin{equation*}
 \sup_{|t|<r} \|\partial^n_{\tilde{m},t}\tilde{\psi}(t)\|_{\EE}\leq A B^n \tilde{m}(n)\quad\textrm{for every}\quad n\in\NN_0.
\end{equation*}
Hence by (\ref{eq:est_nq!}) the formal power series solution $\hat{u}(t,z)$ given by (\ref{eq:v_tilde2}) is convergent for $|t|<r$ and $|z|<(1-q)B^{-1}$, so $u(t,z)\in\Oo(D^2,\EE)$.

Similarly to the proof of Lemma \ref{lem:2}, by the moment Taylor formula for $\tilde{\psi}(t)$ (Proposition \ref{prop:taylor} and Remark \ref{re:l}) we see that
\begin{equation*}
 \frac{\partial^n_{\tilde{m},t}\tilde{\psi}(0)}{\mathfrak{m}(n)[n]_q!}=\frac{\partial^n_{\tilde{m},t}\tilde{\psi}(0)}{\tilde{m}(n)}=\frac{\tilde{\psi}^{(n)}(0)}{n!}\quad\textrm{for every}\quad n\in\NN_0.
\end{equation*}
Hence using (\ref{eq:v_tilde2}) we conclude that
\begin{equation*}
 u(0,z)=\sum_{n=0}^{\infty}\frac{\partial^n_{\tilde{m},t}\tilde{\psi}(0)}{\mathfrak{m}(n)}z^n=\sum_{n=0}^{\infty}\frac{[n]_q!\tilde{\psi}^{(n)}(0)}{n!}z^n.
\end{equation*}
We will follow the proof of necessity in \cite[Theorem 3.1]{I-A} with $\kappa=\nu=1$ and $x_0=0$. By the Cauchy integral formula, we see that
\begin{equation*}
 \varphi(z)=u(0,z)=\frac{1}{2\pi i}\oint_{|\eta|=\rho}\frac{\tilde{\psi}(\eta)}{\eta}\sum_{n=0}^{\infty}(q;q)_n\Big(\frac{z}{(1-q)\eta}\Big)^n\,d\eta
 \end{equation*}
 for sufficiently small $|z|$. By (\ref{eq:q_binomial}) and Heine's transformation formula (Proposition \ref{prop:2}) we obtain 
 \begin{equation*}
  \sum_{n=0}^{\infty}(q;q)_n\Big(\frac{z}{(1-q)\eta}\Big)^n={}_{2}\phi_1\left(\begin{array}{c}
                       q,q\\
                       0
                     \end{array}
;q,\frac{z}{(1-q)\eta}\right)=\frac{\big(q,q\frac{z}{(1-q)\eta};q\big)_{\infty}}{\big(\frac{z}{(1-q)\eta};q\big)_{\infty}}
 \sum_{j=0}^{\infty}\frac{\big(\frac{z}{(1-q)\eta};q\big)_j}{\big(q\frac{z}{(1-q)\eta},q;q\big)_j}q^j.
 \end{equation*}
For fixed $z\not=0$, the function 
\begin{equation*}
\eta\longmapsto\frac{1}{\big(\frac{z}{(1-q)\eta};q\big)_{\infty}}=\prod_{n=0}^{\infty}\frac{\eta}{\eta-(1-q)^{-1}zq^n}
\end{equation*}
is meromorphic on $\CC$ with simple poles at
\begin{equation*}
 \eta=\eta_n(z):=(1-q)^{-1}zq^n\quad\text{for}\quad n\in\NN_0.
\end{equation*}
Using the residue theorem we see that
\begin{equation*}
 \varphi(z)=(q;q)_{\infty}\sum_{n=0}^{\infty}\tilde{\psi}\big( \frac{zq^n}{1-q}\big)\mathop{\mathrm{Res}}_{\eta=\eta_n(z)}\frac{1}{\big(\frac{z}{(1-q)\eta};q\big)_{\infty}}\frac{1}{\eta}(q^{1-n};q)_{\infty}\sum_{j=0}^{\infty}\frac{(q^{-n};q)_j}{(q^{1-n},q;q)_j}q^j.
\end{equation*}
Since $(q^{-n},q)_j=0$ for $j>n$ and $\frac{(q^{1-n};q)_{\infty}}{(q^{1-n};q)_j}=0$ for $j<n$, we get
\begin{equation*}
 (q^{1-n};q)_{\infty}\sum_{j=0}^{\infty}\frac{(q^{-n};q)_j}{(q^{1-n},q;q)_j}q^j
 =(q^{1-n};q)_{\infty}\frac{(q^{-n};q)_n}{(q^{1-n},q;q)_n}q^n
 =(q;q)_{\infty}\frac{(q^{-n};q)_n}{(q;q)_n}q^n.
\end{equation*}
Moreover
\begin{equation*}
 \mathop{\mathrm{Res}}_{\eta=\eta_n(z)}\frac{1}{\big(\frac{z}{(1-q)\eta};q\big)_{\infty}}\frac{1}{\eta}=\frac{1}{(q^{-n};q)_n (q;q)_{\infty}}.
\end{equation*}
Hence
\begin{equation*}
 \varphi(z)=(q;q)_{\infty}\sum_{n=0}^{\infty}\tilde{\psi}\big( \frac{zq^n}{1-q}\big)\frac{q^n}{(q;q)_n}.
\end{equation*}
Since there exist $A,B<\infty$ such that $\|\tilde{\psi}(t)\|_{\EE}\leq Ae^{B|t|^k}$ for every $t\in\hat{S}_d$, we conclude that
\begin{equation*}
 \|\varphi(z)\|_{\EE}\leq Ae^{B(1-q)^{-k}|z|^k} \sum_{n=0}^{\infty}\frac{(q;q)_{\infty}}{(q;q)_n}q^n\leq Ae^{\tilde{B}|z|^k} \sum_{n=0}^{\infty}q^n\leq 
 \tilde{A}e^{\tilde{B}|z|^k}
\end{equation*}
for some positive constants $\tilde{A},\tilde{B}<\infty$ and for every $z\in \hat{S}_d$. It means that
 $\varphi(z)\in\Oo^k(\hat{S}_d,\EE)$.  
\end{proof}
\bigskip\par
We are now ready to prove the main result of the paper
\begin{Thm}[The main theorem]
\label{th:main}
The sequence $([n]_q!)_{n\geq 0}$ for $q\in[0,1)$  preserves summability.
\end{Thm}
\begin{proof}
If $q=0$ then $([n]_q!)_{n\geq 0}=(1)_{n\geq 0}$ preserves summability in a trivial way. So we may assume that $q\in(0,1)$.

First, we will prove that if $\sum_{n=0}^{\infty}x_nt^n\in\EE\{t\}_{k,d}$ then also $\sum_{n=0}^{\infty}\frac{x_n}{[n]_q!}t^n\in\EE\{t\}_{k,d}$.

Fix $k>0$ and $d\in\RR$. We assume that $\hat{x}(t)=\sum_{n=0}^{\infty}x_nt^n\in\EE\{t\}_{k,d}$. Let $\mathfrak{m}=(\mathfrak{m}(n))_{n\geq0}$, where $\mathfrak{m}(u)$ is a moment function of order $1/k$. Then $\psi(t):=\Bo_{\mathfrak{m},t}\hat{x}(t)=\sum_{n=0}^{\infty}\frac{x_n}{\mathfrak{m}(n)}t^n\in\Oo^k(\hat{S}_d,\EE)$. 

Let $v=v(t,z)\in\Oo(D^2,\EE)$ be a solution of the equation 
\begin{equation}
\label{eq:v_thm2}
    \left\{
    \begin{array}{l}
     (\partial_{\mathfrak{m},t}-\partial_{\mathfrak{m},z})v=0\\
     v(t,0)=\psi(t)\in\Oo^k(\hat{S}_d,\EE).
    \end{array}
    \right.,
 \end{equation}
Using Lemma \ref{lem:1} with replaced variables we conclude that $\varphi(z):=v(0,z)=\psi(z)$, so $ \varphi(z)\in\Oo^k(\hat{S}_d,\EE)$.

Now, let $u(t,z)\in\Oo(D^2,\EE)$ be a solution of the initial value problem
\begin{equation}
\label{eq:u_thm2}
    \left\{
    \begin{array}{l}
     (\partial_{\tilde{m},t}-\partial_{\mathfrak{m},z})u=0\\
     u(0,z)=\varphi(z)\in\Oo^k(\hat{S}_d,\EE)
    \end{array}
    \right.,
 \end{equation}
 where $\tilde{m}(n):=\mathfrak{m}(n)[n]_q!$. By Lemma \ref{lem:2} $u(t,0)\in\Oo^k(\hat{S}_d,\EE)$.
 
 Let us recall that $v(t,0)=\psi(t)=\sum_{n=0}^{\infty}\frac{x_n}{\mathfrak{m}(n)}t^n$.
 
 On the other hand, since $v$ is a solution of (\ref{eq:v_thm2}), we have $\partial_{\mathfrak{m},t}v(t,z)=\partial_{\mathfrak{m},z}v(t,z)$, and by the moment Taylor formula in $t$ of $v$ (see Proposition \ref{prop:taylor} and Remark \ref{re:u}), we have
 \begin{equation*}
  v(t,z)=\sum_{n=0}^{\infty}\frac{\partial^n_{\mathfrak{m},t}v(0,z)}{\mathfrak{m}(n)}t^n=\sum_{n=0}^{\infty}\frac{\partial^n_{\mathfrak{m},z}v(0,z)}{\mathfrak{m}(n)}t^n=\sum_{n=0}^{\infty}\frac{\partial^n_{\mathfrak{m},z}\varphi(z)}{\mathfrak{m}(n)}t^n.
 \end{equation*}
 Therefore we have
 \begin{equation*}
   \partial^n_{\mathfrak{m},z}\varphi(0)=x_n\quad\text{for}\quad n\in\NN_0.
 \end{equation*}
Observe that $u(t,z)=\sum_{n=0}^{\infty}\frac{\partial^n_{\mathfrak{m},z}\varphi(z)}{\mathfrak{m}(n)[n]_q!}t^n$ from Proposition \ref{prop:aux0}. Hence
\begin{equation*}
 u(t,0)=\sum_{n=0}^{\infty}\frac{\partial^n_{\mathfrak{m},z}\varphi(0)}{\mathfrak{m}(n)[n]_q!}t^n=\sum_{n=0}^{\infty}\frac{x_n}{\mathfrak{m}(n)[n]_q!}t^n\in\Oo^k(\hat{S}_d,\EE).
\end{equation*}
It means that $\sum_{n=0}^{\infty}\frac{x_n}{[n]_q!}t^n$ is $k$-summable in a direction $d$ from Remark \ref{re:general}.
\bigskip\par
To finish the proof we will show that if 
$\sum_{n=0}^{\infty}\frac{x_n}{[n]_q!}t^n\in\EE\{t\}_{k,d}$
then also $\sum_{n=0}^{\infty}x_nt^n\in\EE\{t\}_{k,d}$.

If $\hat{y}(t)=\sum_{n=0}^{\infty}\frac{x_n}{[n]_q!}t^n\in\EE\{t\}_{k,d}$ then
\begin{equation*}
\tilde{\psi}(t):=\Bo_{\mathfrak{m},t}\hat{y}(t)=\sum_{n=0}^{\infty}\frac{x_n}{\mathfrak{m}(n)[n]_q!}t^n\in\Oo^k(\hat{S}_d,\EE).
\end{equation*}
Let  $u(t,z)\in\Oo(D^2,\EE)$ be a solution of the boundary value problem
\begin{equation*}
    \left\{
    \begin{array}{l}
     (\partial_{\tilde{m},t}-\partial_{\mathfrak{m},z})u=0\\
     u(t,0)=\tilde{\psi}(t)\in\Oo^k(\hat{S}_d,\EE)
    \end{array}
    \right..
 \end{equation*}
 By Lemma \ref{lem:3} we see that $\varphi(z):=u(0,z)\in\Oo^k(\hat{S}_d,\EE)$.
 
 Now, let $v(t,z)\in\Oo(D^2,\EE)$ be a solution of the equation
\begin{equation*}
    \left\{
    \begin{array}{l}
     (\partial_{\mathfrak{m},t}-\partial_{\mathfrak{m},z})v=0\\
     v(0,z)=\varphi(z)\in\Oo^k(\hat{S}_d,\EE)
    \end{array}
    \right..
 \end{equation*}
 By Lemma \ref{lem:1} we get
 \begin{equation*}
  v(t,0)=\varphi(t)\in\Oo^k(\hat{S}_d,\EE).
 \end{equation*}
Since $u$ is a solution of (\ref{eq:u_thm2}), by the moment Taylor formula in $t$ for $u$ (see Proposition \ref{prop:taylor} and Remark \ref{re:u}), we deduce that 
\begin{equation*}
u(t,z)=\sum_{n=0}^{\infty}\frac{\partial_{\tilde{m},t}^nu(0,z)}{\tilde{m}(n)}t^n=\sum_{n=0}^{\infty}\frac{\partial_{\mathfrak{m},z}^nu(0,z)}{\tilde{m}(n)}t^n=\sum_{n=0}^{\infty}\frac{\partial_{\mathfrak{m},z}^n\varphi(z)}{\tilde{m}(n)}t^n,
\end{equation*}
so from $u(t,0)=\tilde{\psi}(t)$, we get $\partial_{\mathfrak{m},z}^n\varphi(0)=x_n$ for $n\in\mathbb{N}_0$.
Finally we conclude that
\begin{equation*}
 \sum_{n=0}^{\infty}\frac{x_n}{\mathfrak{m}(n)}t^n=\sum_{n=0}^{\infty}\frac{\partial^n_{\mathfrak{m},z}\varphi(0)}{\mathfrak{m}(n)}t^n=v(t,0)\in\Oo^k(\hat{S}_d,\EE).
\end{equation*}
It means that $\sum_{n=0}^{\infty}x_nt^n\in\EE\{t\}_{k,d}$.
\end{proof}

\begin{Rem}
We can also prove in a similar way that the sequence $([\kappa n]_q!)_{n\geq0}$ preserves summability for $\kappa\in\mathbb{N}$ and $q\in[0,1)$ in terms of a generalisation of Heine's transformation formula (see \cite[Proposition 3.2]{I-A}).
\end{Rem}

\begin{Rem}
Notice, that Theorem \ref{th:main} shows that the sequence $m=(m(n))_{n\geq0}=([n]_q!)_{n\geq0}$ belongs to the group of sequences preserving summability (see Remark \ref{re:groups}), but we do not know if this sequence is inherited from a moment function $\mathfrak{m}(u)$ of order $0$ in the sense of Definition~\ref{df:moment_general}, i.e. if $([n]_q!)_{n\geq0}$ belongs to the subgroup (\ref{eq:subgroup}).  
\end{Rem}

\section{The Cauchy problem for moment operators of order zero} \label{sec:9} 
In this section we consider the Cauchy problem for the linear equations $P(\partial_{m,t},\partial_z)u=0$ with constant coefficients, where $\partial_{m,t}$ is an operator of order $0$. We will show that if additionally a sequence $m$ preserves summability then summable solutions are characterised in the same way as for the solutions of  $P(\partial_{\mathbf{1},t},\partial_z)u=0$, which is a special case of the equation $P(\partial_{\mathfrak{m}_1,t},\partial_{\mathfrak{m}_2,z})u=0$ already studied in \cite{Mic8} under condition that $\mathfrak{m}_1(u)$ and $\mathfrak{m}_2(u)$ are moment functions of real orders. By the main result of the paper it allows us to characterise summable solutions of general linear $q$-difference-differential equations $P(D_{q,t},\partial_z)u=0$ with constant coefficients.
It gives a far greater generalisation of the results from \cite{I-A}.

We assume that $P(\lambda,\zeta)$ is a general polynomial of two variables of order $p$ with respect to $\lambda$ and $\varphi_j(z)\in\Oo(D)$ for $j=0,\dots,p-1$.

We study the relation between the solution $\hat{u}(t,z)\in\Oo(D)[[t]]$ of the Cauchy problem
\begin{equation}
\label{eq:CP_u}
    \left\{
    \begin{array}{l}
     P(\partial_{m,t},\partial_z)u=0\\
     \partial^j_{m,t}u(0,z)=\varphi_j(z),\ j=0,\dots,p-1,
    \end{array}
    \right.
 \end{equation}
and the solution $\hat{v}(t,z)\in\Oo(D)[[t]]$ of the similar initial value problem
\begin{equation}
 \label{eq:CP_tilde_u}
    \left\{
    \begin{array}{l}
     P(\partial_{\mathbf{1},t},\partial_z)v=0\\
     \partial^j_{\mathbf{1},t}v(0,z)=\varphi_j(z),\ j=0,\dots,p-1.
    \end{array}
    \right.
 \end{equation}
 
 First, let us observe that
 \begin{Prop}
  \label{prop:uv}
  A formal power series $\hat{u}(t,z)=\sum_{n=0}^{\infty}\frac{u_n(z)}{m(n)}t^n$ is a solution of (\ref{eq:CP_u}) if and only if $\hat{v}(t,z)=\sum_{n=0}^{\infty}u_n(z)t^n$ is a formal power series solution of (\ref{eq:CP_tilde_u}).
 \end{Prop}
\begin{proof}
 ($\Rightarrow$) Let $\hat{u}(t,z)=\sum_{n=0}^{\infty}\frac{u_n(z)}{m(n)}t^n$ be a formal solution of (\ref{eq:CP_u}). Using the commutation formula (Proposition \ref{prop:commutate}) $\Bo_{m^{-1},t}\partial_{m,t} = \partial_{\mathbf{1},t}\Bo_{m^{-1},t}$ with $m^{-1}=(m(n)^{-1})_{n\geq 0}$ and applying the Borel transform $\Bo_{m^{-1},t}$ to the Cauchy problem (\ref{eq:CP_u}) we conclude that $\hat{v}(t,z)=\Bo_{m^{-1},t}\hat{u}(t,z)$ is a formal solution of (\ref{eq:CP_tilde_u}).
\medskip\par
($\Leftarrow$) The proof is analogous. It is sufficient to apply the Borel transform $\Bo_{m,t}$ to the Cauchy problem (\ref{eq:CP_tilde_u}) and to observe that $\hat{u}(t,z)=\Bo_{m,t}\hat{v}(t,z)$.
\end{proof}

 Using the above proposition and the properties of the sequences preserving Gevrey order and summability we conclude that
\begin{Cor}
\label{cor:equivalence_gevrey}
Let $P(\lambda,\zeta)$ be a polynomial of two variables of order $p$ with respect to $\lambda$ and $s\in\RR$. We also assume that a sequence $m=(m(n))_{n\geq 0}$ preserves Gevrey order.

Then a formal power series solution $\hat{u}(t,z)\in\Oo(D)[[t]]$ of the Cauchy problem (\ref{eq:CP_u}) is of Gevrey order $s$ if and only if a power series solution $\hat{v}(t,z)=\Bo_{m^{-1},t}\hat{u}(t,z)$ of the Cauchy problem (\ref{eq:CP_tilde_u}) is of the same Gevrey order.
\end{Cor}

\begin{Cor}
\label{cor:equivalence_summability}
Let $P(\lambda,\zeta)$ be a polynomial of two variables of order $p$ with respect to $\lambda$, $k>0$ and $d\in\RR$. We also assume that a sequence $m=(m(n))_{n\geq 0}$ preserves summability.

Then a formal power series solution $\hat{u}(t,z)\in\Oo(D)[[t]]$ of the Cauchy problem (\ref{eq:CP_u}) is $k$-summable in a direction $d$ if and only if a power series solution $\hat{v}(t,z)=\Bo_{m^{-1},t}\hat{u}(t,z)$ of the Cauchy problem (\ref{eq:CP_tilde_u}) is $k$-summable in the same direction.
\end{Cor}

\begin{Cor}
\label{cor:equivalence_multisummability}
Let $P(\lambda,\zeta)$ be a polynomial of two variables of order $p$ with respect to $\lambda$, $\mathbf{k}=(k_1,\dots,k_N)$ with $k_1>\cdots>k_N>0$ and $\mathbf{d}=(d_1,\dots,d_N)\in\RR^N$ be an
 admissible multidirection. We also assume that a sequence $m=(m(n))_{n\geq 0}$ preserves summability.

Then a formal power series solution $\hat{u}(t,z)\in\Oo(D)[[t]]$ of the Cauchy problem (\ref{eq:CP_u}) is $\mathbf{k}$-multisummable in a multidirection $\mathbf{d}$ if and only if a power series solution $\hat{v}(t,z)=\Bo_{m^{-1},t}\hat{u}(t,z)$ of the Cauchy problem (\ref{eq:CP_tilde_u}) is $\mathbf{k}$-multisummable in the same multidirection.
\end{Cor}

In the case $m=([n]_q!)_{n\geq 0}$ for $q\in[0,1)$ 
we can rewrite (\ref{eq:CP_u}) as the Cauchy problem for the general homogeneous linear $q$-difference-differential equation with constant coefficients
\begin{equation}
\label{eq:CP_u_q}
    \left\{
    \begin{array}{l}
     P(D_{q,t},\partial_z)u=0\\
     D^j_{q,t}u(0,z)=\varphi_j(z)\in\Oo(D),\ j=0,\dots,p-1,
    \end{array}
    \right.
 \end{equation}
 In this special case by Theorem \ref{th:main} we may formulate Proposition \ref{prop:uv} and Corollaries \ref{cor:equivalence_gevrey}, \ref{cor:equivalence_summability} and \ref{cor:equivalence_multisummability} as
\begin{Cor}
 Let $P(\lambda,\zeta)$ be a polynomial of two variables of order $p$ with respect to $\lambda$ and $q\in[0,1)$.  We also assume that $\hat{u}(t,z)=\sum_{n=0}^{\infty}\frac{u_n(z)}{[n]_q!}t^n$ and $\hat{v}(t,z)=\sum_{n=0}^{\infty}u_n(z)t^n$ are formal power series belonging to the space $\Oo(D)[[t]]$. Then the following equivalences hold:
 \begin{enumerate}
  \item $\hat{u}(t,z)$ is a formal power series solution of (\ref{eq:CP_u_q}) if and only if $\hat{v}(t,z)$ is a formal power series solution of (\ref{eq:CP_tilde_u}).
  \item Fix $s\in\RR$. $\hat{u}(t,z)$ is a formal power series solution of (\ref{eq:CP_u_q}) of Gevrey order $s$ if and only if $\hat{v}(t,z)$ is a formal power series solution of (\ref{eq:CP_tilde_u}) of the same Gevrey order $s$.
  \item Fix $k>0$ and $d\in\RR$. $\hat{u}(t,z)$ is a formal power series solution of (\ref{eq:CP_u_q}) that is $k$-summable in a direction $d$ if and only if $\hat{v}(t,z)$ is a formal power series solution of (\ref{eq:CP_tilde_u}) that is $k$-summable in the same direction.
  \item Fix $\mathbf{k}=(k_1,\dots,k_N)$ with $k_1>\cdots>k_N>0$ and an admissible multidirection $\mathbf{d}=(d_1,\dots,d_N)\in\RR^N$. $\hat{u}(t,z)$ is a formal power series solution of (\ref{eq:CP_u_q}) that is  $\mathbf{k}$-multisummable in a multidirection $\mathbf{d}$ if and only if $\hat{v}(t,z)$ is a formal power series solution of (\ref{eq:CP_tilde_u}) that is $\mathbf{k}$-multisummable in the same multidirection.
 \end{enumerate}

\end{Cor}

\bigskip\par
Let $\lambda(\zeta)$ be an algebraic function on $\CC$. It means that there exists a polynomial $P(\lambda,\zeta)$ of two complex variables such that the function $\lambda(\zeta)$ satisfies equation $P(\lambda(\zeta),\zeta)=0$. By the implicit function theorem the function $\lambda(\zeta)$ is holomorphic on $\CC$ except at a finite number of singular or branching points. Moreover this function has a moderate growth at infinity. More precisely there exist a \emph{pole order at infinity}
$\tilde{q}\in\QQ$ and a \emph{leading term} $\lambda\in\CC^*$
such that 
$$\lim_{\zeta\to\infty}\frac{\lambda(\zeta)}{\zeta^{\tilde{q}}}=\lambda.$$
We denote it shortly by $\lambda(\zeta)\sim\lambda\zeta^{\tilde{q}}$.

Hence there exists $r_0<\infty$ and $\kappa\in\NN$ such that $\lambda(\zeta)$ is a holomorphic function of the variable $\xi=\zeta^{1/\kappa}$ for $|\zeta|>r_0$ with a pole at infinity. It means that the function
$\xi\mapsto \lambda(\xi^{\kappa})$ has the Laurent series expansion $\lambda(\xi^{\kappa})=\sum_{j=-n}^{\infty}\frac{a_j}{\xi^j}$ at infinity for some coefficients $a_j\in\CC$ with $a_{-n}=\lambda$ and $n=\tilde{q}\kappa\in\ZZ$. This expansion is convergent for $|\xi|>r_0^{1/\kappa}$ with a pole of order $n$ at infinity.

For such functions we may define the following pseudodifferential operators

\begin{Def}[see {\cite[Definition 13]{Mic8}}]
Let $\lambda(\zeta)$ be a holomorphic function of the variable $\xi=\zeta^{1/\kappa}$ for $|\zeta|\geq r_0$ (for some $\kappa\in\NN$ and $r_0>0$) and of moderate growth at infinity. A~\emph{moment pseudodifferential operator} $\lambda(\partial_{z})\colon \Oo_{1/\kappa}(D)\to\Oo_{1/\kappa}(D)$ is defined by
\begin{gather}
\label{eq:lambda}
\lambda(\partial_{z})\varphi(z):=\frac{1}{2\kappa\pi i}\oint^{\kappa}_{|w|=\varepsilon}\varphi(w)\int_{e^{i\theta}r_0}^{e^{i\theta}\infty}\lambda(\zeta)\mathbf{E}_{1/\kappa}(z^{1/\kappa}\zeta^{1/\kappa})e^{-(\zeta w)}\,d\zeta dw
\end{gather}
for every $\varphi(z)\in\Oo_{1/\kappa}(D_r)$ and $|z|<\varepsilon<r$, where $\theta\in(-\arg w -\frac{\pi}{2}, -\arg w +\frac{\pi}{p})$, $\mathbf{E}_{1/\kappa}(z):=\sum_{n=0}^{\infty}\frac{z^n}{\Gamma(1+n/\kappa)}$ is the Mittag-Leffler function of index $1/\kappa$ and $\oint^{\kappa}_{|w|=\varepsilon}$ means that we integrate $\kappa$ times along the positively oriented circle of radius $\varepsilon$.
\end{Def}

\begin{Rem}
  The right-hand side of (\ref{eq:lambda}) does not depend on the choice of the number $r_0$ such that $\lambda(\zeta)$ is holomorphic for $|\zeta|\geq r_0$ (see \cite[Proposition 2]{Mic10}). The value of $\lambda(\partial_z)\varphi(z)$ depends only on $\varphi(z)$ and on the behaviour of the algebraic function $\lambda(\zeta)$ at a neighbourhood of infinity.
\end{Rem}

If $P(\lambda,\zeta)$ is a general polynomial of two variables of order $p$ with respect to $\lambda$ given in (\ref{eq:CP_u}) and (\ref{eq:CP_tilde_u}), then we may write it as
\begin{equation*}
P(\lambda,\zeta)=P_0(\zeta)\lambda^p-\sum_{j=1}^p P_j(\zeta)\lambda^{p-j}
\end{equation*} 
for some polynomials $P_0(\zeta),\dots,P_p(\zeta)$ of one variable.

If $P_0(\zeta)\neq\textrm{const.}$ then a formal solution of (\ref{eq:CP_u}) (and of (\ref{eq:CP_tilde_u})) is not uniquely determined. To avoid this inconvenience we choose some special solution which is called the \emph{normalized formal solution} (see \cite{B5} and \cite{Mic5}). To this end we factorise the polynomial $P(\lambda,\zeta)$
as
 \begin{equation*}
P(\lambda,\zeta)=P_0(\zeta)\prod_{j=1}^n\prod_{l=1}^{p_j}(\lambda-\lambda_{jl}(\zeta))^{p_{jl}}=:P_0(\zeta)\tilde{P}(\lambda,\zeta),
 \end{equation*} 
where $P_0(\zeta)\sim a_0\zeta^{q_0}$ for some $a_0\in\CC\setminus\{0\}$ and $q_0\in\NN_0$,
$\sum_{j=1}^n\sum_{l=1}^{p_j}p_{jl}=p$ and $\lambda_{jl}(\zeta)$ are the roots of the characteristic equation
   $P(\lambda,\zeta)=0$ satisfying $\lambda_{jl}(\zeta)\sim \lambda_{jl}\zeta^{q_j}$ for some $\lambda_{jl}\in\CC^*$ and $q_j\in\QQ$, i.e. $q_j=\mu_j/\nu_j$ for some relatively prime $\mu_j\in\ZZ$ and $\nu_j\in\NN$.

Since $\lambda_{jl}(\partial_z)$ are defined by (\ref{eq:lambda}), also the moment pseudodifferential operator
\begin{equation*}
\tilde{P}(\partial_{m,t},\partial_z)=\prod_{j=1}^n\prod_{l=1}^{p_j}(\partial_{m,t}-\lambda_{jl}(\partial_z))^{p_{jl}}
\end{equation*}
is well defined.

Now we are ready to define the uniquely determined normalized solution of (\ref{eq:CP_u}) (resp. of (\ref{eq:CP_tilde_u})).

\begin{Def}
 A formal solution $\hat{u}$ of (\ref{eq:CP_u}) (resp. $\hat{v}$ of (\ref{eq:CP_tilde_u})) is called the \emph{normalized formal solution} if $\hat{u}$ (resp. $\hat{v}$) is also a solution of the pseudodifferential equation $\tilde{P}(\partial_{m,t},\partial_z)u=0$ (resp. $\tilde{P}(\partial_{\mathbf{1},t},\partial_z)v=0$).
\end{Def}

\begin{Thm}
  \label{th:gevrey}
  Let $\hat{u}$ be a normalised formal solution of (\ref{eq:CP_u}) and $m=(m(n))_{n\geq 0}$ be a sequence preserving Gevrey order.
  
   Then 
   \begin{equation}
   \label{eq:decomposition}
   \hat{u}=\sum_{j=1}^n\sum_{l=1}^{p_j}\sum_{\alpha=1}^{p_{jl}}\hat{u}_{jl\alpha}
   \end{equation}
   with
   $\hat{u}_{jl\alpha}$ being a formal solution of a simple pseudodifferential equation
   \begin{equation}
   \label{eq:u_jl}
    \left\{
    \begin{array}{l}
     (\partial_{m,t}-\lambda_{jl}(\partial_{z}))^{\alpha}u_{jl\alpha}=0\\
     \partial_{m,t}^{\beta} u_{jl\alpha}(0,z)=0\ \  (\beta=0,\dots,\alpha-2)\\
     \partial_{m,t}^{\alpha-1} u_{jl\alpha}(0,z)=\lambda_{jl}^{\alpha-1}(\partial_{z})\varphi_{jl\alpha}(z)\in\Oo_{1/\kappa}(D),
    \end{array}
    \right.
   \end{equation}
   where $\varphi_{jl\alpha}(z):=\sum_{\beta=0}^{p-1}d_{jl\alpha\beta}(\partial_{z})
   \varphi_{\beta}(z)$ and $d_{jl\alpha\beta}(\zeta)$ are some holomorphic
   functions of the variable $\xi=\zeta^{1/\kappa}$ and of moderate growth.
   \par
   Moreover, a formal solution $\hat{u}_{jl\alpha}\in\Oo_{1/\kappa}(D)[[t]]$ is a Gevrey series of order $\max\{q_j,0\}$ with respect to $t$.
   \par
   If additionally the initial data $\varphi_{\beta}(z)\in\Oo^{k}(\CC)$ for $\beta=0,\dots,p-1$ then a formal solution $\hat{u}_{jl\alpha}\in\Oo_{1/\kappa}(D)[[t]]$ is a Gevrey series of order $\max\{q_j,0\}(1-1/k)$ with respect to $t$. 
\end{Thm}
\begin{proof}
 Since $m$ is a sequence preserving Gevrey order,
 by Corollary \ref{cor:equivalence_gevrey} 
 it is sufficient to prove the statement for the sequence $m=\mathbf{1}=(1)_{n\in\NN}$, which is inherited from the moment function $\mathbf{1}(u)\equiv 1$ for $u\geq 0$.
 Applying \cite[Theorem 1]{Mic8} with $\mathfrak{m}_1(u)\equiv 1$, $\mathfrak{m}_2(u)=\Gamma(1+u)$, $s_1=0$, $s_2=1$ and $s=0$ we get the decomposition (\ref{eq:decomposition}) with $\hat{u}_{jl\alpha}$ satisfying (\ref{eq:u_jl}) and being a Gevrey series of order $\max\{q_j,0\}$. If additionally $\varphi_{\beta}(z)\in\Oo^k(\CC)$ then $\hat{\varphi}_{\beta}(z)\in\CC[[z]]_{-1/k}$ for $\beta=0,\dots,p-1$. In this case we also apply \cite[Theorem 1]{Mic8} but with $s=-1/k$, and we conclude that $\hat{u}_{jl\alpha}$ is a Gevrey series of order $\max\{q_j,0\}(1-1/k)$.
\end{proof}
\bigskip\par
To show the result for summable and multisummable solutions, 
additionally we may assume that $q_1>q_2>\dots>q_n$ and
\begin{equation}
\label{eq:tilde_n}
\tilde{n}:=
\left\{
   \begin{array}{lll}
 0&\textrm{for}& q_1\leq 0\\
\max\{i\colon q_i>0\}&\textrm{for}& q_1>0.
\end{array}
\right.
\end{equation}

First observe that if $\tilde{n}=0$ and the sequence $m=(m(n))_{n\geq 0}$ preserves Gevrey order, then by Theorem \ref{th:gevrey}, the normalized formal solution $\hat{u}$ of (\ref{eq:CP_u}) is convergent.

Now, let us assume that $\tilde{n}=1$. In this case we will study summable solutions of (\ref{eq:CP_u}). Namely, we have
\begin{Thm}
  \label{th:summability}
  Under the above conditions, we assume that $\tilde{n}=1$, $d\in\RR$ and the sequence $m=(m(n))_{n\geq 0}$ preserves summability. 
  
  If $\varphi_j(z)\in\Oo^1(\hat{S}_{(d+\arg\lambda_{1l}+2n\pi)/q_1})$ for $j=0,\dots,p-1$, $l=1,\dots,p_{1}$ and $n=0,\dots,\mu_1-1$ then a normalised formal solution $\hat{u}$ of (\ref{eq:CP_u}) is $1/q_1$-summable in a direction $d$.
  
  In the opposite side, let us assume additionally that the initial data in (\ref{eq:CP_u}) satisfy $\varphi_0(z)=\dots=\varphi_{p-2}(z)=0$ and $\varphi_{p-1}(z)=\varphi(z)\in\Oo(D)$. If a normalized formal solution $\hat{u}$ of (\ref{eq:CP_u}) is $1/q_1$-summable in a direction $d$
  then $\varphi(z)\in\Oo^1(\hat{S}_{(d+\arg\lambda_{1l}+2n\pi)/q_1)})$ for $l=1,\dots,p_1$ and $n=0,\dots,\mu_1-1$.
\end{Thm}
\begin{proof}
Since $m$ is a sequence preserving summability,
 by Corollary \ref{cor:equivalence_summability} 
 it is sufficient to prove the statement for the sequence $m=\mathbf{1}=(1)_{n\in\NN}$, which is inherited from the moment function $\mathbf{1}(u)\equiv 1$ for $u\geq 0$.
 By Theorem \ref{th:gevrey} we get the decomposition of $\hat{u}$ given by (\ref{eq:decomposition}).
 Moreover, since $q_j\leq 0$ for $j=2,\dots,n$ by the same theorem we conclude that 
 \begin{equation*}
 \hat{u}_2:=\sum_{j=2}^n\sum_{l=1}^{p_j}\sum_{\alpha=1}^{p_{jl}}\hat{u}_{jl\alpha}\in\Oo_{1/\kappa}(D)[[t]]_0.
 \end{equation*}
 It means that its sum $u_2$ is convergent, hence also $\hat{u}_2\in\Oo_{1/\kappa}(D)\{t\}_{1/q_1,d}$.
 \par
 Fix $l\in\{1,\dots,p_1\}$ and $\alpha\in\{1,\dots,p_{1l}\}$. Since $\varphi_{1l\alpha}\in\Oo_{1/\kappa}^1(\hat{S}_{(d+\arg\lambda_{1l}+2n\pi)/q_1})$ for $n=0,\dots,\mu_1-1$, where $\alpha=1,\dots,p_{1l}$ and $l=1,\dots,p_1$, by \cite[Theorem 3]{Mic8}
 with $\mathfrak{m}_1(u)\equiv 1$, $\mathfrak{m}_2(u)=\Gamma(1+u)$, $s_1=0$, $s_2=1$ and $s=0$ we obtain $\hat{u}_{1l\alpha}\in\Oo_{1/\kappa}(D)\{t\}_{1/q_1,d}$. Hence we see that
 \begin{equation*}
  \hat{u}_1:=\sum_{l=1}^{p_1}\sum_{\alpha=1}^{p_{1l}}\hat{u}_{1l\alpha}\in\Oo_{1/\kappa}(D)\{t\}_{1/q_1,d}.
 \end{equation*}
It means that also $\hat{u}=\hat{u}_1+\hat{u}_2\in\Oo_{1/\kappa}(D)\{t\}_{1/q_1,d}$. Additionally $\hat{u}\in\Oo(D)[[t]]$, so finally $\hat{u}\in\Oo(D)\{t\}_{1/q_1,d}$.

The proof of the theorem in the opposite side proceeds along the same line as the proof of \cite[Theorem 6]{Mic7}. We fix $l\in\{1,\dots,p_1\}$ and we define
$\hat{u}_{1l}:=\tilde{P}_{1l}(\partial_{\mathbf{1},t},\partial_z)\hat{u}$, where
\begin{equation*}
 \tilde{P}_{1l}(\lambda,\zeta):=\tilde{P}(\lambda,\zeta)/(\lambda-\lambda_{1l}(\zeta)).
\end{equation*}
Since $\hat{u}\in\Oo(D)\{t\}_{1/q_1,d}$ and $\hat{u}_{1l}=\tilde{P}_{1l}(\partial_{\mathbf{1},t},\partial_z)\hat{u}$, we conclude that also $\hat{u}_{1l}\in\Oo_{1/\kappa}(D)\{t\}_{1/q_1,d}$.
On the other hand, since
\begin{equation*}
 (\partial_{\mathbf{1},t}-\lambda_{1l}(\partial_z))\tilde{P}_{1l}(\partial_{\mathbf{1},t},\partial_z)=\tilde{P}(\partial_{\mathbf{1},t}, \partial_z)
\end{equation*}
and $\hat{u}_{1l}(0,z)=\partial^{p-1}_{\mathbf{1},t}\hat{u}(0,z)=\varphi(z)\in\Oo(D)$, we get that $\hat{u}_{1l}$ is a formal solution of the Cauchy problem
  \begin{equation*}
    \left\{
    \begin{array}{l}
     (\partial_{\mathbf{1},t}-\lambda_{1l}(\partial_{z}))u_{1l}=0\\
     u_{1l}(0,z)=\varphi(z)\in\Oo(D).
    \end{array}
    \right.
   \end{equation*}
Hence by \cite[Theorem 4]{Mic8} with $\mathfrak{m}_1(u)\equiv 1$, $\mathfrak{m}_2(u)=\Gamma(1+u)$, $s_1=0$, $s_2=1$ and $s=0$ we conclude that $\varphi(z)\in\Oo^1(\hat{S}_{(d+\arg\lambda_{1l}+2n\pi)/q_1})$ for $n=0,\dots,\mu_1-1$, which completes the proof.
\end{proof}

In the case, when $\tilde{n}\geq 2$ it is natural to study multisummability of the solution $\hat{u}$ of (\ref{eq:CP_u}). In general the sufficient condition for the multisummability of $\hat{u}$ given in terms of the analytic properties of the initial data is not necessary, since the multisummability of $\hat{u}$ satisfying (\ref{eq:decomposition}) does not imply the summability of $\hat{u}_{jl\alpha}$ (see \cite[Example 2]{Mic7}). For this reason, as in \cite[Definition 11]{Mic7}, we define a special kind of multisummability for which that implication holds.
\begin{Def}
 Let $(d_{\tilde{n}},\dots,d_{1})\in\RR^{\tilde{n}}$ be an admissible multidirection with respect to $(1/q_{\tilde{n}},\dots,1/q_1)$. We say that $\hat{u}$ is \emph{$(1/q_{\tilde{n}},\dots,1/q_1)$-multisummable in the multidirection $(d_{\tilde{n}},\dots,d_{1})$ with respect to the decomposition (\ref{eq:decomposition})} if $\hat{u}_{jl\alpha}$ is $1/q_j$ summable in a direction $d_j$ (for $j=1,\dots,\tilde{n}$) or is convergent (for $j=\tilde{n}+1,\dots,n$), where $l=1,\dots,p_j$, $\alpha=1,\dots,p_{jl}$.  
\end{Def}

We have
\begin{Thm}
  \label{th:multisummability}
  Under the above conditions, we assume that $\tilde{n}>1$, $(d_{\tilde{n}},\dots,d_{1})\in\RR^{\tilde{n}}$ is an admissible multidirection with respect to $(1/q_{\tilde{n}},\dots,1/q_1)$ and the sequence $m=(m(n))_{n\geq 0}$ preserves summability. 
  
  If $\varphi_{\alpha}(z)\in\Oo^1(\hat{S}_{(d_{j}+\arg\lambda_{jl}+2n\pi)/q_j)})$ for $\alpha=0,\dots,p-1$, $l=1,\dots,p_j$, $n=0,\dots,\mu_j-1$ and $j=1,\dots,\tilde{n}$ then a normalised formal solution $\hat{u}$ of  (\ref{eq:CP_u}) is $(1/q_{\tilde{n}},\dots,1/q_1)$-multisummable in a multidirection $(d_{\tilde{n}},\dots,d_{1})$.
  
  In the opposite side, let us assume additionally that the initial data in (\ref{eq:CP_u}) satisfy $\varphi_0(z)=\dots=\varphi_{p-2}(z)=0$ and $\varphi_{p-1}(z)=\varphi(z)\in\Oo(D)$. If a normalized formal solution $\hat{u}$ of (\ref{eq:CP_u}) 
  is $(1/q_{\tilde{n}},\dots,1/q_1)$-multisummable in a multidirection $(d_{\tilde{n}},\dots,d_{1})$ with respect to the decomposition (\ref{eq:decomposition})
  then $\varphi(z)\in\Oo^1(\hat{S}_{(d_{j}+\arg\lambda_{jl}+2n\pi)/q_j)})$ for $l=1,\dots,p_j$, $n=0,\dots,\mu_j-1$ and $j=1,\dots,\tilde{n}$.
\end{Thm}
\begin{proof}
 Since $m$ is a sequence preserving summability,
 by Corollary \ref{cor:equivalence_multisummability} 
 it is sufficient to prove the statement for the sequence $m=\mathbf{1}=(1)_{n\in\NN}$, which is inherited from the moment function $\mathbf{1}(u)\equiv 1$.
 
 If $\varphi_{\alpha}(z)\in\Oo^1(\hat{S}_{(d_{j}+\arg\lambda_{jl}+2n\pi)/q_j)})$ for $\alpha=0,\dots,p-1$, $l=1,\dots,p_j$, $n=0,\dots,\mu_j-1$ and $j=1,\dots,\tilde{n}$ then applying \cite[Theorem 4]{Mic8}  with $\mathfrak{m}_1(u)\equiv 1$, $\mathfrak{m}_2(u)=\Gamma(1+u)$, $s_1=0$, $s_2=1$ and $s=0$ we conclude that a normalised formal solution $\hat{u}$ of (\ref{eq:CP_u}) is $(1/q_{\tilde{n}},\dots,1/q_1)$-multisummable in a multidirection $(d_{\tilde{n}},\dots,d_{1})$.
 
 In the opposite side, we assume that a normalized formal solution $\hat{u}$ of (\ref{eq:CP_u})
  is $(1/q_{\tilde{n}},\dots,1/q_1)$-multisummable in a multidirection $(d_{\tilde{n}},\dots,d_{1})$ with respect to the decomposition (\ref{eq:decomposition}), and additionally that the initial data in (\ref{eq:CP_u}) satisfy $\varphi_0(z)=\dots=\varphi_{p-2}(z)=0$ and $\varphi_{p-1}(z)=\varphi(z)\in\Oo(D)$. Then by \cite[Theorem 5]{Mic8} with $\mathfrak{m}_1(u)\equiv 1$, $\mathfrak{m}_2(u)=\Gamma(1+u)$, $s_1=0$, $s_2=1$ and $s=0$ we conclude that
  $\varphi(z)\in\Oo^1(\hat{S}_{(d_{j}+\arg\lambda_{jl}+2n\pi)/q_j)})$ for $l=1,\dots,p_j$, $n=0,\dots,\mu_j-1$ and $j=1,\dots,\tilde{n}$.
\end{proof}

At the end we also find the sufficient and necessary conditions for the convergence of the formal solution $\hat{u}$ of (\ref{eq:CP_u})

\begin{Thm}
\label{th:convergent}
 If the sequence $m=(m(n))_{n\geq 0}$ preserves Gevrey order and the initial data $\varphi_j(z)$ are entire functions of exponential growth of order $1$ for $j=0,\dots,p-1$, then the formal solution $\hat{u}$ of (\ref{eq:CP_u}) is convergent.
 
 In the opposite side, we assume that the sequence $m=(m(n))_{n\geq 0}$ preserves summability, the initial data of (\ref{eq:CP_u}) satisfy conditions $\varphi_0(z)=\dots=\varphi_{p-2}(z)=0$ and $\varphi_{p-1}(z)=\varphi(z)\in\Oo(D)$ and the number $\tilde{n}$ defined by (\ref{eq:tilde_n}) is positive. Under these assumptions if the formal solution $\hat{u}$ of (\ref{eq:CP_u}) is convergent then $\varphi(z)$ is the entire function of exponential growth of order $1$.
\end{Thm}
\begin{proof}
Applying Theorem \ref{th:gevrey} with $k=1$ we conclude that $\hat{u}$ satisfies the decomposition (\ref{eq:decomposition}) with
$\hat{u}_{jl\alpha}\in\Oo_{1/\kappa}(D)[[t]]_0$. Then also $\hat{u}\in\Oo(D)[[t]]_0$. This means that $\hat{u}$ is convergent.

To prove the second part of the theorem, observe that the condition $\tilde{n}>0$ means that $q_1>0$. 

Similarly to the proof of Theorem \ref{th:summability} (see also the proof of \cite[Theorem 6]{Mic7}) we take
$\hat{u}_{11}:=\tilde{P}_{11}(\partial_{m,t},\partial_z)\hat{u}$, where
$$
\tilde{P}_{11}(\lambda,\zeta):=\tilde{P}(\lambda,\zeta)/(\lambda-\lambda_{11}(\zeta)).
$$
Since $\hat{u}$ is convergent and $\hat{u}_{11}=\tilde{P}_{11}(\partial_{m,t},\partial_z)\hat{u}$, we get that $\hat{u}_{11}\in\Oo_{1/\kappa}(D)[[t]]$ is also convergent. It means that $\hat{u}_{11}\in\Oo_{1/\kappa}(D)\{t\}_{1/q_1,d}$ for every $d\in\RR$.
On the other hand, since
\begin{equation*}
 (\partial_{m,t}-\lambda_{11}(\partial_z))\tilde{P}_{11}(\partial_{m,t},\partial_z)=\tilde{P}(\partial_{m,t}\partial_z)
\end{equation*}
and $\hat{u}_{11}(0,z)=\partial^{p-1}_{m,t}\hat{u}(0,z)=\varphi(z)\in\Oo(D)$, we get that $\hat{u}_{11}$ is a formal solution of the Cauchy problem
  \begin{equation}
   \label{eq:u_11}
    \left\{
    \begin{array}{l}
     (\partial_{m,t}-\lambda_{11}(\partial_{z}))u_{11}=0\\
     u_{11}(0,z)=\varphi(z)\in\Oo(D).
    \end{array}
    \right.
   \end{equation}
Since the sequence $m$ preserves summability, by Corollary \ref{cor:equivalence_summability} we conclude that $\hat{v}_{11}:=\Bo_{m^{-1},t}\hat{u}_{11}\in\Oo_{1/\kappa}(D)\{t\}_{1/q_1,d}$ for every $d\in\RR$ and $\hat{v}_{11}$ is a formal solution of (\ref{eq:u_11}) with $\partial_{m,t}$ replaced by $\partial_{\mathbf{1},t}$.  
Hence, applying \cite[Theorem 4]{Mic8} with $\mathfrak{m}_1(u)\equiv 1$, $\mathfrak{m}_2(u)=\Gamma(1+u)$, $s_1=0$, $s_2=1$ and $s=0$ to this new equation we conclude that $\varphi(z)\in\Oo^1(\hat{S}_{(d+\arg\lambda_{1l}+2n\pi)/q_1})$ for $n=0,\dots,\mu_1-1$ and for every $d\in\RR$, and consequently $\varphi\in\Oo^1(\CC)$, which completes the proof.
\end{proof}

\begin{Rem}
 Observe that by Theorem \ref{th:main}, the assertions of Theorems \ref{th:gevrey}--\ref{th:convergent} hold in a particular case of $q$-difference-differential equation (\ref{eq:CP}).
\end{Rem}

\begin{Rem}
 In the similar way, using \cite{Mic9} instead of \cite{Mic8},
 we may get the characterisation of summable or multisummable formal solutions $\hat{u}$ to the inhomogeneous Cauchy problems
 \begin{equation}
 \label{eq:inhomo}
\left\{
    \begin{array}{l}
     P(D_{q,t},\partial_z)u=\hat{f}(t,z),\\  
     D^j_{q,t}u(0,z)=\varphi_j(z)\ \text{for}\ j=0,\dots,p-1,
     \end{array}
    \right.
\end{equation}
 in terms of the properties of the inhomogeneity $\hat{f}(t,z)\in\Oo(D)[[t]]$ and the Cauchy data $\varphi_j(z)\in\Oo(D)$, $j=0,\dots,p-1$.
 
 Equivalently, we may find conditions under which the $q$-integro-differential operator $P(D_{q,t},\partial_z)D_{q,t}^{-p}$ is a linear automorphism on the space of Gevrey or summable or multisummable series.
 
 We are not going to discuss about summable solutions of the inhomogeneous $q$-difference-differential equation (\ref{eq:inhomo}) in this article,
 but in future it could be a good starting point to study also some non-linear $q$-difference-differential equations from the 
 point of view of theory of summability and multisummability. 
\end{Rem}

\end{document}